\documentclass[11pt,leqno]{article}
\usepackage[margin=2.5cm]{geometry} % page geometry
\usepackage{microtype} % typographical improvements

\usepackage{url}
\usepackage[pagebackref]{hyperref}
\usepackage{amsfonts}
\usepackage{amsthm}
\usepackage{amsmath}
\usepackage{amscd}
\usepackage{amssymb}
\usepackage{caption,subcaption}
\usepackage[shortlabels]{enumitem}
\usepackage{mathtools}
\usepackage{tikz}
\usepackage{tikz-cd}
\usetikzlibrary{arrows,arrows.meta}
\tikzset{>=latex}
\tikzcdset{arrow style=tikz, diagrams={>=latex}}

\def\FF {{\mathbb F}}     %% field
\def\QQ {{\mathbb Q}}     %% rationals
\def\RR {{\mathbb R}}     %% real numbers
\def\ZZ {{\mathbb Z}}     %% integers

\def\ring#1{\ifmmode \mathaccent'027 #1\else \rm\accent'027 #1\fi}

\def\mc {\mathcal}

\def\Ob {\mathrm{Ob}}
\def\Hom {\mathrm{Hom}}
\def\Aut {\mathrm{Aut}}
\def\Ext {\mathrm{Ext}}

\mathchardef\mhyphen="2D

\newtheorem{theorem}{Theorem}[section]

\newtheorem{lemma}[theorem]{Lemma}
\newtheorem{prop}[theorem]{Proposition}
\newtheorem{coro}[theorem]{Corollary}
\theoremstyle{plain}
\newtheorem{remark}[theorem]{Remark}
\theoremstyle{definition}
\newtheorem{df}[theorem]{Definition}

\theoremstyle{plain}

\numberwithin{equation}{section}

\newlength{\miniwidth}

\newcommand{\Addresses}{{% additional braces for segregating \footnotesize
  \bigskip
  \footnotesize

	(F.~Haiden) \textsc{University of Oxford, Mathematical Institute, Andrew Wiles Building, Woodstock Road, Oxford OX2 6GG, UK} \par\nopagebreak
	\textit{E-mail:} \texttt{Fabian.Haiden@maths.ox.ac.uk}
	\medskip
	
}}

\begin{document}

\title{Flags and tangles}
\author{Fabian Haiden}
\date{}
\maketitle

\begin{abstract}
We show that two constructions yield equivalent braided monoidal categories.
The first is topological, based on Legendrian tangles and skein relations, while the second is algebraic, in terms of chain complexes with complete flag and convolution-type products.
The category contains Iwahori--Hecke algebras of type $A_n$ as endomorphism algebras of certain objects.
\end{abstract}

\setcounter{tocdepth}{2}
\tableofcontents

\section{Introduction}

Braided monoidal categories~\cite{joyal_street93} provide a unifying framework for low-dimensional topology and representation theory.
Categories of tangles play a fundamental role, being freely generated among braided monoidal categories with certain extra properties~\cite{freyd_yetter89}.
Linearizations of the category of tangles, the passage from classical topology to ``quantum topology'', come from HOMFLY-PT skein theory. 
The novelty of this work is to consider \textit{legendrian} tangles and skein relations to obtain an analog of Turaev's Hecke category~\cite{turaev90}.
Our main result provides a very different description of this category in terms of complete flags on chain complexes and convolution-type products.

\subsection*{Prelude: Iwahori--Hecke algebras of type $A_n$.}

To motivate our later considerations we begin by recalling several points of view (algebraic and topological) on Iwahori--Hecke algebras of type $A_n$, denoted by $\mc H_n$.
More details may be found in several places, e.g. the textbook~\cite{kassel_turaev}.

Suppose we want to find a $q$-deformation of (the group algebra of) the symmetric group $S_{n+1}$.
The $q$-analog of a set $X$ with $n+1$ elements is a vector space $V$ of dimension $n+1$ over a finite field $\FF_q$, and the $q$-analog of a total order on $X$ is a complete flag on $V$.
Let $\QQ Fl(V)$ be the $\QQ$-vector space with basis the set of complete flags on $V$. 
The $q$-analog of the transposition $(i,i+1)\in S_{n+1}$ is the linear operator $T_i$ on $\QQ Fl(V)$, $i=1,\ldots,n$, which takes a complete flag
\[
0=V_0\subset V_1\subset \ldots\subset V_i\subset\ldots\subset V_{n+1}=V
\]
to the sum of the $q$ flags
\[
0=V_0\subset V_1\subset \ldots\subset V_i'\subset\ldots\subset V_{n+1}=V
\]
where $V_i'$ is a subspace with $V_{i-1}\subset V'\subset V_{i+1}$ and $V_i'\neq V_i$.
Thus, $T_i$ is a sort of \textit{random modification at the $i$-th step}.
It is easy to see that the relations
\begin{align}
T_iT_{i+1}T_i&=T_{i+1}T_iT_{i+1},\qquad i=1,\ldots,n-1 \label{hr1} \\
T_iT_j&=T_jT_i,\qquad\qquad\quad |i-j|>1 \label{hr2} \\
T_i^2&=(q-1)T_i+q,\qquad i=1,\ldots,n \label{hr3} 
\end{align}
hold, and these turn out to be a complete set of relations for the subalgebra of $\mathrm{End}(\QQ Fl(V))$ generated by the $T_i$.
Since the relations are polynomial in $q$, we can also treat $q$ as a formal parameter and define $\mc H_n$ over $\ZZ[q^{\pm 1}]$, as is usually done.
Setting $q=1$ gives the group algebra of the symmetric group $S_n$, and the definition by generators and relations generalizes to arbitrary Coxeter groups.

The fact that \eqref{hr1} and \eqref{hr2} are the defining relations of the braid group suggests that we represent elements of $\mc H_n$ by braid diagrams with $n+1$ strands, where $T_i$ corresponds to the braid with a single positive crossing of the $i$-th and $i+1$-st strand.
The relation \eqref{hr3} is then equivalent, modulo the second Reidemeister move, to the \textit{skein relation}
\begin{equation*}
\begin{tikzpicture}[baseline=-\dimexpr\fontdimen22\textfont2\relax,scale=.6]
\draw[thick,->] (1,-1) to (-1,1);
\draw[white,line width=2mm] (-1,-1) to (1,1);
\draw[thick,->] (-1,-1) to (1,1);
\draw[dashed] (0,0) circle [radius=1.414];
\end{tikzpicture}%
-q\;
\begin{tikzpicture}[baseline=-\dimexpr\fontdimen22\textfont2\relax,scale=.6]
\draw[thick,->] (-1,-1) to (1,1);
\draw[white,line width=2mm] (-1,1) to (1,-1);
\draw[thick,->] (1,-1) to (-1,1);
\draw[dashed] (0,0) circle [radius=1.414];
\end{tikzpicture}%
=(q-1)\;
\begin{tikzpicture}[baseline=-\dimexpr\fontdimen22\textfont2\relax,scale=.6]
\draw[thick,->] (-1,-1) to (-.62,-.62) to [out=45,in=-45] (-.62,.62) to (-1,1);
\draw[thick,->] (1,-1) to (.62,-.62) to [out=135,in=-135] (.62,.62) to (1,1);
\draw[dashed] (0,0) circle [radius=1.414];
\end{tikzpicture}%
\end{equation*}
which explains the relevance of $\mc H_n$ to knot theory, c.f. Jones~\cite{jones87}.

The purpose of this work is to extend the above from vector spaces to $\ZZ$-graded chain complexes of vector spaces.
Instead of defining a single algebra, it is natural to construct a braided monoidal category $\mc H$, see below.
It turns out that this category has a topological interpretation, though somewhat surprisingly we need to use Legendrian tangles instead of ordinary (topological) tangles.
The use of Legendrian tangles was suggested by the work by the author, where homomorphisms are constructed from Legendrian skein algebras to Hall algebras of Fukaya categories~\cite{haiden_skeinhall}.
This text is, for the most part, logically independent from \cite{haiden_skeinhall}.
We anticipate that the braided monoidal categories defined here will be useful in gluing ``frozen'' variants of the skein algebras and Hall algebras from \cite{haiden_skeinhall}.

\subsection*{The monoidal category $\mc H$}

Let us first describe the construction in terms of flags on chain complexes.
Fix a finite ground field $\mathbf k=\FF_q$, then the monoidal category $\mc H$ is defined as follows.
\begin{itemize}
\item
\textbf{Objects:} Finite-dimensional $\ZZ$-graded vector spaces $V=\bigoplus_{i\in\ZZ}V_i$ over $\mathbf k$ together with a complete flag of graded subspaces $F_iV\subseteq V$, $i=0,\ldots,\dim V$, $\dim F_iV=i$.
\item
\textbf{Monoidal product of objects:} $(V,F_\bullet V)\otimes(W,F_\bullet W)$ has underlying graded vector space $V\oplus W$ and complete flag
\begin{equation}\label{flag_on_direct_sum}
0=F_0V\subset F_1V\subset\cdots\subset F_{\dim V}V\subset V\oplus F_1W\subset\ldots\subset V\oplus F_{\dim W}W=V\oplus W.
\end{equation}
\item
\textbf{Morphisms:} $\Hom_{\mc H}(V,W)$ is the $\QQ$-vector space with basis the set $B(V,W)$ of equivalence classes of triples $(d_V,f,d_W)$ where $d_V$ is a differential (of degree 1) on $V$ with $d_V(F_iV)\subseteq F_{i-1}V$, $d_W$ is a differential on $W$ with $d_W(F_iW)\subseteq F_{i-1}W$, and $f:V\to W$ is a quasi-isomorphism.
Two triples $(d_V,f,d_W)$ and $(d_V',f',d_W')$ are equivalent if there are isomorphisms of graded vector spaces $\phi:V\to V'$, $\psi:W\to W'$ which preserve the flags and differentials, and such that $\psi\circ f$ and $f'\circ\phi$ are homotopic (equivalently: induce the same map on cohomology).
\item
\textbf{Composition:}
If $U,V,W\in\mathrm{Ob}(\mc H)$ and $(d_U,f,d_V)$ represents a morphism $U\to V$ and $(d_V',g,d_W)$ represents a morphism $V\to W$ then their composition is defined as
\begin{equation}
\label{h_comp}
(d_V',g,d_W)\circ(d_U,f,d_V):= \left(\prod_{i=1}^{\infty}\left|\Hom^{-i}_{< 0}(V,V)\right|^{(-1)^i}\right)\sum_b\left(d_U,gbf,d_W\right)
\end{equation}
where $\Hom^{-i}_{< 0}(V,V)$ denotes homogeneous maps of degree $-i$ sending $F_jV$ to $F_{j-1}V$, and the sum runs over all isomorphisms of chain complexes with complete flag, $b:(V,d_V)\to (V,d_V')$, i.e. preserving the grading, the differential, and the flag.
Extend this product bilinearly to all morphisms.
The origin of this formula is discussed in Subsection~\ref{subsec_comp_conceptual}.
\item
\textbf{Monoidal product of morphisms:}
This is essentially the Hall algebra product for dg-categories, see Subsection~\ref{subsec_hall}.
If $U,V,X,Y\in\mathrm{Ob}(\mc H)$, $(d_U,f,d_V)$ represents a morphism $U\to V$, and $(d_X,g,d_Y)$ represents a morphism $X\to Y$ then
\begin{align*}
(d_U,f,d_V)\otimes (d_X,g,d_Y):=&\left(\prod_{i=0}^{\infty}\left(\left|\Hom^{-i}(X,U)\right|\left|\Hom^{-i}(Y,V)\right|\left|\Hom^{-i-1}(X,V)\right|\right)^{(-1)^{i+1}}\right) \\
&\quad\cdot\sum_{\delta}\left(\begin{bmatrix} d_U & \delta_{11} \\ 0 & d_X \end{bmatrix},\begin{bmatrix} f & \delta_{12} \\ 0 & g \end{bmatrix},\begin{bmatrix} d_V & \delta_{22} \\ 0 & d_Y \end{bmatrix} \right)\nonumber
\end{align*}
where the sum is over 
\begin{equation*}
\delta=\left(\delta_{11},\delta_{12},\delta_{22}\right)\in\Hom^1(X,U)\oplus\Hom^0(X,V)\oplus\Hom^1(Y,V)
\end{equation*}
with
\begin{gather}
d_U\delta_{11}+\delta_{11}d_X=0,\qquad d_V\delta_{22}+\delta_{22}d_Y=0,\qquad \delta_{12}d_X+f\delta_{11}=d_V\delta_{12}+\delta_{22}g.
\end{gather}
\end{itemize}

To relate this to the Iwahori--Hecke algebra of type $A_n$ we note that for the object of $\mc H$ given by a vector space $V=\mathbf k^{n+1}$ concentrated in degree zero and with the standard complete flag we have 
\[
B(V,V)=B\setminus GL(n+1,\mathbf k)/B=S_{n+1}
\]
where $B$ is the subgroup of upper triangular matrices and the second equality is the Bruhat decomposition.
Moreover, the algebra $\Hom_{\mc H}(V,V)$ is just $\mc H_n$, specialized to a prime power $q$.

\begin{prop}
\label{prop_h_monoidal}
$\mc H$ as defined above is a monoidal category.
\end{prop}

The proof of this proposition is completed in Subsection~\ref{subsec_h_monoidal}.
The equivalence with the category $\mc S|_q$ (see below) shows that $\mc H$ has a natural braiding.

\subsection*{The monoidal category $\mc S$}

The braided monoidal category $\mc S$ is defined in terms of graded Legendrian tangles.
Before giving the definition we recall some basic terminology from Legendrian knot theory.
A Legendrian curve $L$ in $\RR^3$ is an embedded curve which is everywhere tangent to the contact distribution $\mathrm{Ker}(dz-ydx)$.
The \textit{front projection} of $L$ is the image under the projection to the $(x,z)$-plane.
For generic compact Legendrian curve the front projection is locally the graph of a smooth function $z(x)$ except near a finite set of singular points which are either simple crossings $(\times)$, left cusps $(\prec)$, or right cusps $(\succ)$.
A \textit{graded Legendrian tangle} in $J^1[0,1]=[0,1]\times\RR^2\subset\RR^3$ is a compact graded Legendrian curve $L$ in $[0,1]\times\RR^2$ with boundary in $\{0,1\}\times\RR\times\{0\}$.
A \textit{grading} or \textit{Maslov potential} for a Legendrian curve is given, in terms of the front projection, by a labeling of the strands connecting cusps by integers such that at each cusp the number assigned to the lower strand is one greater than the number assigned to the upper strand.

We now come to the definition of $\mc S$.

\begin{itemize}
\item
\textbf{Objects:} Finite $\ZZ$-graded subsets $X$ of $\RR$ up to isotopy (A $\ZZ$-grading is simply a function $\deg:X\to\ZZ$.)
\item
\textbf{Monoidal product of objects:} $X\otimes Y$ is obtained by stacking $Y$ on top of $X$, which is well defined up to isotopy.
\item
\textbf{Morphisms:} $\Hom_{\mc S}(X,Y)$ is the module over $\ZZ[q^{\pm},(q-1)^{-1}]$ generated by isotopy classes of tangles $L$ with $\partial_0L=Y$ and $\partial_1L=X$ modulo the following skein relations, where $\delta_{m,n}$ is the Kronecker delta.
These are relations between tangles which are identical except inside a small ball where they differ as shown.
\begin{equation}\label{gs_1f}
\begin{tikzpicture}[baseline=-\dimexpr\fontdimen22\textfont2\relax,scale=.6]
\draw[thick] (-1.414,0) to [out=0,in=180] (0,-.5) to [out=0,in=180] (1.414,0);
\draw[line width=1mm,white] (1,1) to [out=-135,in=0] (0,0) to [out=0,in=135] (1,-1);
\draw[thick] (1,1) to [out=-135,in=0] (0,0) to [out=0,in=135] (1,-1);
\node[blue,above] at (-.8,-.2) {\scriptsize{n}};
\node[blue,left] at (.95,.85) {\scriptsize{m-1}};
\node[blue,left] at (.95,-.85) {\scriptsize{m}};
\draw[dashed] (0,0) circle [radius=1.414];
\end{tikzpicture}%
-q^{(-1)^{m-n}}\;
\begin{tikzpicture}[baseline=-\dimexpr\fontdimen22\textfont2\relax,scale=.6]
\draw[thick] (1,1) to [out=-135,in=0] (0,0) to [out=0,in=135] (1,-1);
\draw[line width=1mm,white] (-1.414,0) to [out=0,in=180] (0,.5) to [out=0,in=180] (1.414,0);
\draw[thick] (-1.414,0) to [out=0,in=180] (0,.5) to [out=0,in=180] (1.414,0);
\node[blue,below] at (-.8,.2) {\scriptsize{n}};
\node[blue,left] at (.95,.85) {\scriptsize{m-1}};
\node[blue,left] at (.95,-.85) {\scriptsize{m}};
\draw[dashed] (0,0) circle [radius=1.414];
\end{tikzpicture}%
=\delta_{m,n}(q-1)\;
\begin{tikzpicture}[baseline=-\dimexpr\fontdimen22\textfont2\relax,scale=.6]
\draw[thick] (-1.414,0) to [out=0,in=135] (1,-1);
\draw[thick] (1,1) to [out=-135,in=0] (0,.3) to [out=0,in=180] (1.414,0);
\node[blue,below] at (-.5,0) {\scriptsize{n}};
\node[blue,left] at (.95,.85) {\scriptsize{m-1}};
\node[blue,left] at (.95,-.05) {\scriptsize{m}};
\draw[dashed] (0,0) circle [radius=1.414];
\end{tikzpicture}%
-\delta_{m,n+1}(1-q^{-1})\;
\begin{tikzpicture}[baseline=-\dimexpr\fontdimen22\textfont2\relax,scale=.6]
\draw[thick] (-1.414,0) to [out=0,in=-135] (1,1);
\draw[thick] (1,-1) to [out=135,in=0] (0,-.3) to [out=0,in=180] (1.414,0);
\node[blue,below] at (-.7,.2) {\scriptsize{n}};
\node[blue,above] at (.8,-.2) {\scriptsize{m-1}};
\node[blue,left] at (.95,-.85) {\scriptsize{m}};
\draw[dashed] (0,0) circle [radius=1.414];
\end{tikzpicture}%
\tag{S1}
\end{equation}
\begin{equation}\label{gs_2f}
\begin{tikzpicture}[baseline=-\dimexpr\fontdimen22\textfont2\relax,scale=.6]
\draw[thick] (-1,0) to [out=0,in=180] (0,.6) to [out=0,in=180] (1,0) to [out=180,in=0] (0,-.6) to [out=180,in=0] (-1,0);
\draw[dashed] (0,0) circle [radius=1.414];
\end{tikzpicture}
=(q-1)^{-1}\;
\begin{tikzpicture}[baseline=-\dimexpr\fontdimen22\textfont2\relax,scale=.6]
\draw[dashed] (0,0) circle [radius=1.414];
\end{tikzpicture}
\tag{S2}
\end{equation}
\begin{equation}\label{gs_3f}
\begin{tikzpicture}[baseline=-\dimexpr\fontdimen22\textfont2\relax,scale=.6]
\draw[thick] (-1.414,0) to [out=0,in=-135] (.5,.5) to [out=-135,in=45] (-.5,-.5) to [out=45,in=180] (1.414,0);
\draw[dashed] (0,0) circle [radius=1.414];
\end{tikzpicture}
=0
\tag{S3}
\end{equation}
(The vertically flipped variant of \eqref{gs_1f} already follows from \eqref{gs_1f} and legendrian isotopy.)
\item
\textbf{Composition:} Horizontal composition (concatenation) of tangles
\item
\textbf{Monoidal product of morphisms:} Vertical composition (stacking) of tangles
\item
\textbf{Braiding:} The braiding morphism $X\otimes Y\to Y\otimes X$ is given by the following tangle (with grading determined by $X$ and $Y$).
We note that this is not an isomorphism in general, so gives a braiding only in a weak sense, i.e. representing the positive braid monoid only instead of the full braid group.
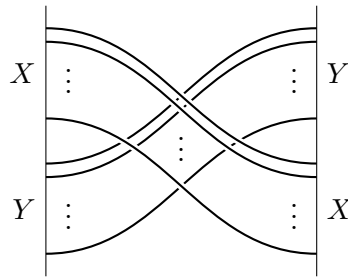
\begin{equation*}
\begin{tikzpicture}[scale=.6]
\draw[thick] (-3,-.5) to [out=0,in=180] (3,2.5);
\draw[thick] (-3,-.8) to [out=0,in=180] (3,2.2);
\draw[thick] (-3,-2.5) to [out=0,in=180] (3,.5);
\draw[white,line width=1mm] (-3,.5) to [out=0,in=180] (3,-2.5);
\draw[white,line width=1mm] (-3,2.5) to [out=0,in=180] (3,-.5);
\draw[white,line width=1mm] (-3,2.2) to [out=0,in=180] (3,-.8);
\draw[thick] (-3,.5) to [out=0,in=180] (3,-2.5);
\draw[thick] (-3,2.5) to [out=0,in=180] (3,-.5);
\draw[thick] (-3,2.2) to [out=0,in=180] (3,-.8);
\node at (-2.5,1.5) {$\vdots$};
\node at (-2.5,-1.5) {$\vdots$};
\node at (2.5,1.5) {$\vdots$};
\node at (2.5,-1.5) {$\vdots$};
\node at (0,0) {$\vdots$};
\draw (-3,-3) to (-3,3);
\draw (3,-3) to (3,3);
\node[left] at (-3,-1.5) {$Y$};
\node[left] at (-3,1.5) {$X$};
\node[right] at (3,1.5) {$Y$};
\node[right] at (3,-1.5) {$X$};
\end{tikzpicture}
\end{equation*}
\end{itemize}

The following proposition is automatic.

\begin{prop}
\label{prop_s_monoidal}
$\mc S$ as defined above is a braided monoidal category.
\end{prop}

The local description of front projections of Legendrian curves immediately implies that morphisms in $\mc S$ are generated by the following elementary tangles under horizontal and vertical composition.

\begin{equation}
\label{elem_tangles}
\begin{tikzpicture}[baseline=-\dimexpr\fontdimen22\textfont2\relax,scale=.6]
\draw[thick] (-1,0) to (1,0);
\node[above,blue] at (0,-.1) {\scriptsize{n}};
\node at (0,-1.5) {$1_n$};
\end{tikzpicture}
\qquad\qquad
\begin{tikzpicture}[baseline=-\dimexpr\fontdimen22\textfont2\relax,scale=.6]
\draw[thick] (1,.5) to [out=180,in=0] (-.5,0) to [out=0,in=180] (1,-.5);
\node[right,blue] at (.9,.5) {\scriptsize{n}};
\node[right,blue] at (.9,-.5) {\scriptsize{n+1}};
\node at (.25,-1.5) {$\lambda_n$};
\end{tikzpicture}
\qquad\qquad
\begin{tikzpicture}[baseline=-\dimexpr\fontdimen22\textfont2\relax,scale=.6]
\draw[thick] (-1,.5) to [out=0,in=180] (.5,0) to [out=180,in=0] (-1,-.5);
\node[left,blue] at (-.9,.5) {\scriptsize{n}};
\node[left,blue] at (-.9,-.5) {\scriptsize{n+1}};
\node at (-.25,-1.5) {$\rho_n$};
\end{tikzpicture}
\qquad\qquad
\begin{tikzpicture}[baseline=-\dimexpr\fontdimen22\textfont2\relax,scale=.6]
\draw[thick] (-1,-.5) to [out=0,in=180] (1,.5);
\draw[line width=1mm,white] (-1,.5) to [out=0,in=180] (1,-.5);
\draw[thick] (-1,.5) to [out=0,in=180] (1,-.5);
\node[right,blue] at (.9,.5) {\scriptsize{m}};
\node[right,blue] at (.9,-.5) {\scriptsize{n}};
\node at (0,-1.5) {$\sigma_{m,n}$};
\end{tikzpicture}
\end{equation}

A remark about conventions: 
In Legendrian knot theory tangles are drawn horizontally, with left and right boundary, as opposed to vertically.
The composition of morphisms $f:X\to Y$ and $g:Y\to Z$ is usually written in the order $g\circ f$.
Given tangles $f,g$, if we want $g\circ f$ to be the tangle with $g$ on the left and $f$ on the right, then the source of the tangle must be on the right and the target on the left, which is the convention adopted here.

\subsection*{Main result}

Fix a finite field $\mathbf k=\FF_q$ and consider $\QQ$ as an algebra over $R:=\ZZ[t^{\pm},(t-1)^{-1}]$ via the homomorphism $R\to\QQ$ which sends the formal variable $t$ to the prime power $q$.
Let $\mc S|_q$ be the category obtained from $\mc S$ by base change to $\QQ$.
Thus, $\mathrm{Ob}(\mc S|_q)=\mathrm{Ob}(\mc S)$ and morphisms in $\mc S|_q$ are $\QQ$-linear combinations of tangles modulo skein relations in which $q$ is a prime power.
A $\QQ$-linear functor of monoidal categories from $\mc S|_q$ is determined by its value on objects and elementary tangles. 
We claim that there is a functor $\Phi:\mc S|_q\to \mc H$ defined on objects by mapping a graded subset $X\subset \RR$ to the graded vector space $V$ with basis the elements of $X$, and filtration $F_iV:=\mathrm{Span}(x_1,\ldots,x_i)$ where $x_i$ is the $i$-th (smallest) element of $X$, and on elementary tangles as follows:
\begin{equation*}
\begin{tikzpicture}[baseline=-\dimexpr\fontdimen22\textfont2\relax,scale=.6]
\draw[thick] (-1,0) to (1,0);
\node[above,blue] at (0,-.1) {\scriptsize{n}};
\end{tikzpicture}
\qquad \mapsto\qquad 1_{\mathbf k[-n]}=(q-1)^{-1}(0,1,0)\in\Hom_{\mc H}(\mathbf k[-n],\mathbf k[-n])
\end{equation*}
\begin{equation*}
\begin{tikzpicture}[baseline=-\dimexpr\fontdimen22\textfont2\relax,scale=.6]
\draw[thick] (1,.5) to [out=180,in=0] (-.5,0) to [out=0,in=180] (1,-.5);
\node[right,blue] at (.9,.5) {\scriptsize{n}};
\node[right,blue] at (.9,-.5) {\scriptsize{n+1}};
\end{tikzpicture}
\qquad \mapsto\qquad (q-1)^{-1}\left(
\begin{bmatrix}
0 & 1\\ 0 & 0
\end{bmatrix}
,0,0\right)\in\Hom_{\mc H}(\mathbf k[-n-1]\oplus\mathbf k[-n],0)
\end{equation*}
\begin{equation*}
\begin{tikzpicture}[baseline=-\dimexpr\fontdimen22\textfont2\relax,scale=.6]
\draw[thick] (-1,.5) to [out=0,in=180] (.5,0) to [out=180,in=0] (-1,-.5);
\node[left,blue] at (-.9,.5) {\scriptsize{n}};
\node[left,blue] at (-.9,-.5) {\scriptsize{n+1}};
\end{tikzpicture}
\qquad \mapsto\qquad (q-1)^{-1}\left(
0,0,\begin{bmatrix}
0 & 1\\ 0 & 0
\end{bmatrix}
\right)\in\Hom_{\mc H}(0,\mathbf k[-n-1]\oplus\mathbf k[-n])
\end{equation*}
\begin{equation*}
\begin{tikzpicture}[baseline=-\dimexpr\fontdimen22\textfont2\relax,scale=.6]
\draw[thick] (-1,-.5) to [out=0,in=180] (1,.5);
\draw[line width=1mm,white] (-1,.5) to [out=0,in=180] (1,-.5);
\draw[thick] (-1,.5) to [out=0,in=180] (1,-.5);
\node[right,blue] at (.9,.5) {\scriptsize{m}};
\node[right,blue] at (.9,-.5) {\scriptsize{n}};
\end{tikzpicture}
\qquad \mapsto\qquad (q-1)^{-2}\left(0,
\begin{bmatrix} 0 & 1\\ 1 & 0 \end{bmatrix},0\right)
\in\Hom_{\mc H}(\mathbf k[-n]\oplus\mathbf k[-m],\mathbf k[-m]\oplus\mathbf k[-n])
\end{equation*}

Note that the shift of a graded vector space is by definition/convention given by $(V[n])^k:=V^{k+n}$, so $\mathbf k[-n]$ is concentrated in degree $n$.
The following theorem is the main result of this paper.

\begin{theorem}
\label{thm_phi_equivalence}
The above rules define a monoidal functor $\Phi:\mc S|_q\to \mc H$ which is an equivalence of categories.
\end{theorem}

As a consequence of this theorem we get positive bases on the Hom-spaces in $\mc S$, since Hom-spaces in $\mc H$ have natural positive bases.
It would be interesting to have a version of this theorem for $\ZZ/2$-graded tangles/chain complexes, since in this case the Legendrian skein relations are specializations of the framed HOMFLY-PT skein relations.
The definition of $\mc S$ is completely analogous, but the definition of $\mc H$ is not clear in this case.
This is related to the problem of defining the Hall algebra for $\ZZ/2$-periodic categories.

\subsection*{Outline}

In section~\ref{sec_H} we check that $\mc H$ is a monoidal category and study some of its properties.
The goal of Section~\ref{sec_S} is to obtain a more explicit description of the Hom-spaces in $\mc S$ using ideas from Legendrian knot theory.
The heart of the paper is Section~\ref{sec_functor}, in which we check that $\Phi$ is well-defined, i.e. invariant under Reidemeister moves and skein moves, and provides an equivalence $\mc S|_q\to \mc H$.
Finally, Section~\ref{sec_qrep} describes a dg-model for the category of complexes of quiver representations, while Section~\ref{sec_hcard} provides a more conceptual derivation of the formulas used in the definition of $\mc H$ in terms of homotopy cardinality.

\subsection*{Acknowledgments}
 
We thank Tom Bridgeland, Ben Cooper, Bernhard Keller, Dan Rutherford, Peter Samuelson, and Ivan Smith for stimulating discussions. 
The author presented an early version of the results in May 2019 at the conference on \textit{Interactions between Representation Theory and Homological Mirror Symmetry} at the University of Leicester and greatly benefited from interactions with other participants there.
We also thank an anonymous referee for carefully reviewing the manuscript and suggesting numerous improvements to the exposition.

\section{The category $\mc H$}
\label{sec_H}

This section is devoted to the category $\mc H$.
We begin with some remarks on chain complexes with complete flag in Subsection~\ref{subsec_flags}.
Subsection~\ref{subsec_h_monoidal} contains the proof that $\mc H$ is a monoidal category. 
The most involved step is to show that the monoidal product of morphisms is bifunctorial.
Two dualities for $\mc H$, corresponding to rotation in the category of tangles, are investigated in Subsection~\ref{subsec_h_dualities}.
The Hom-spaces of $\mc H$ are studied in more detail in Subsection~\ref{subsec_cones}.

\subsection{Chain complexes with complete flag}
\label{subsec_flags}

Fix  a base field $\mathbf k$ throughout.
Let $V$ be a $\ZZ$-graded vector space with complete flag $F_\bullet V$ of graded subspaces, i.e. on object of $\mc H$.
We are interested in the classification of differentials $d:V\to V[1]$, $d^2=0$, up to automorphisms of $V$ which preserve the grading and the flag.
This follows essentially the Bruhat decomposition and is contained in work of Barannikov~\cite{barannikov94}.
To state the result we introduce some terminology.

\begin{df}
Let $X$ be a finite, totally ordered, $\ZZ$-graded set.
A \textbf{partial ruling} on $X$ is given by a subset $D\subset X$ and an injective function $\delta:D\to X\setminus D$ such that $\delta(i)<i$ and $\deg\delta(i)=\deg i+1$ for all $i\in D$.
A partial ruling is a \textbf{ruling} if $X=D\cup \delta(D)$.
\end{df}

A $\ZZ$-graded vector space $V$ with complete flag $F_\bullet V$ is classified by the set $\Phi^{-1}V:=\{1,\ldots,n\}$, $n:=\dim V$, with the usual total order and grading such that $i\in X$ has the same degree as the one-dimensional space $F_iV/F_{i-1}V$.
Choose a homogeneous basis $b_1,\ldots,b_n$ of $V$ such that $b_1,\ldots,b_i$ span $F_iV$. 
A partial ruling $(D,\delta)$ on $\Phi^{-1}V$ determines a differential $d(D,\delta)$ on $V$ with $d(b_i)=b_{\delta(i)}$ if $i\in D$ and $d(b_i)=0$ otherwise.
For different choices of bases as above the resulting differentials are conjugate by an automorphism of $V$ preserving the grading and the flag.
The following is an easy consequence of the classical Bruhat decomposition, see~\cite{barannikov94} for details.

\begin{prop}\label{prop_dclass}
Let $V$ be a $\ZZ$-graded vector space with complete flag $F_\bullet V$, then the assignment $(D,\delta)\mapsto d(D,\delta)$ determines a one-to-one correspondence between partial rulings on $\Phi^{-1}V$ on the one hand, and flag-preserving differentials on $V$ up to conjugation by automorphism preserving the grading and the flag on the other.
\end{prop}

Fix a $\ZZ$-graded vector space $V$ with complete flag, $n=\dim V$, and a ruling $(D,\delta)$ on $\Phi^{-1}V$.
An automorphism of $V$ which preserves the grading, the flag, and the differential $d=d(D,\delta)$ is given by an $n\times n$ upper-triangular matrix $A$ with entries $a_{ij}$ where $a_{ij}=0$ if $\deg i\neq \deg j$, $a_{ii}\neq 0$, and
\[
\sum_ja_{ij}d_{jk}=\sum_jd_{ij}a_{jk}
\]
for all $i,k$.
The left-hand side above is $a_{i,\delta(k)}$ if $k\in D$ and $i\leq \delta(k)$ and vanishes otherwise, and similarly the right-hand side is $a_{\delta^{-1}(i),k}$ if $i\in\delta(D)$ and $\delta^{-1}(i)\leq k$ and vanishes otherwise.
It follows that $Ad=dA$ is equivalent to the following list of relations.
\begin{enumerate}[1)]
\item
$a_{i,\delta(k)}=0$ if $k\in D$, $i<\delta(k)$, $\deg i=\deg k+1$, and $i\notin\delta(D)$ or ($i\in\delta(D)$ and $\delta^{-1}(i)>k$).
\item
$a_{\delta^{-1}(i),k}=0$ if $i\in \delta(D)$, $\delta^{-1}(i)<k$, $\deg i=\deg k+1$, and $k\notin D$ or ($k\in D$ and $i>\delta(k)$).
\item
$a_{i,\delta(k)}=a_{\delta^{-1}(i),k}$ if $k\in D$, $i\in\delta(D)$, $i<\delta(k)$, $\delta^{-1}(i)<k$, and $\deg i=\deg k+1$.
\item
$a_{\delta(k),\delta(k)}=a_{kk}$ if $k\in D$.
\end{enumerate}
Picture pairs $(i,k)$ as positions in an $n\times n$-matrix, then the first three types of relations correspond to those pairs $(i,k)$ which are above or to the right of a pair of the form $(\delta(j),j)$.
Thus, as an abstract algebraic variety, the group $\mathrm{Aut}(V,d)$ of automorphisms of $V$ preserving the grading, the flag, and the differential $d$ is 
\[
\left(\mathbb{A}^1_{\mathbf k}\setminus 0\right)^{n-r}\times \mathbb{A}_{\mathbf k}^{m-s}
\]
where
\begin{gather*}
r=r(D,\delta):=|D|, \\
s=s(D,\delta):=\left|\left\{(i,j)\mid \deg(i)=\deg(j)+1,j\in D,\delta(j)>i\text{ or }i\in\delta(D),\delta^{-1}(i)<j\right\}\right|,\\
m=\left|\left\{(i,j)\mid i<j, \deg(i)=\deg(j) \right\}\right|.
\end{gather*}

\subsection{The proof of Proposition~\ref{prop_h_monoidal}}
\label{subsec_h_monoidal}

The composition \eqref{h_comp} is easily seen to be well-defined and associative.

\begin{lemma}
The identity $1_V$ of $V\in\mc H$ is given by
\begin{equation}\label{h_id}
1_V=(q-1)^{-\dim V}\left(\prod_{i=0}^{\infty}\left|\Hom^{-i}_{<0}(V,V)\right|^{(-1)^{i+1}}\right)\sum_d(d,1,d)
\end{equation}
where the sum is over all $d\in\Hom^1_{<0}(V,V)$ with $d^2=0$.
\end{lemma}

\begin{proof}
Let $(d_V,f,d_W)$ represent a morphism $V\to W$.
We will show that $(d_V,f,d_W)\circ 1_V=(d_V,f,d_W)$, the proof that $1_W\circ (d_V,f,d_W)=(d_V,f,d_W)$ being similar.
The set of pairs $(b,d)$ where $d\in\Hom^1_{<0}(V,V)$ is a differential and $b:(V,d)\to (V,d_V)$ is an isomorphism of filtered chain complexes is in bijection, via projection to the first factor, with the subgroup of invertible elements in $\Hom^0_{\leq 0}(V,V)$, since $d=b^{-1}d_Vb$. 
This subgroup has size $(q-1)^{\dim V}|\Hom^0_{<0}(V,V)|$.
Moreover, all triples $(b^{-1}d_Vb,fb,d_W)$ are equivalent to $(d_V,f,d_W)$ in $B(V,W)$, so
\[
 (d_V,f,d_W)\circ\sum_d(d,1,d) =(q-1)^{\dim V}\left(\prod_{i=0}^{\infty}\left|\Hom^{-i}_{<0}(V,V)\right|^{(-1)^{i}}\right)(d_V,f,d_W)
\]
thus $(d_V,f,d_W)\circ 1_V=(d_V,f,d_W)$ in $\mc H$.
\end{proof}

Recall the definition of the monoidal product from the introduction.
The condition on $\delta=(\delta_{11},\delta_{12},\delta_{22})$ can be stated more conceptually as the requirement of being a closed morphism of degree 1 from $((X,d_X),g,(Y,d_Y))$ to $((U,d_U),f,(V,d_V))$ in the derived dg-category of the $A_2$ quiver, $\mc D(A_2)$, as defined in Subsection~\ref{subsec_hall} in the appendix.
In fact, the monoidal product is the Hall algebra product for this dg-category (or rather its full subcategory of perfect objects), hence well-definedness and associativity are general facts. 
Another consequence is the following cohomological formula for the monoidal product.

\begin{lemma}\label{lem_monoidal_ext}
With the notation as in the definition of the monoidal product, the morphism
\[
T(\delta):=\left(\begin{bmatrix} d_U & \delta_{11} \\ 0 & d_X \end{bmatrix},\begin{bmatrix} f & \delta_{12} \\ 0 & g \end{bmatrix},\begin{bmatrix} d_V & \delta_{22} \\ 0 & d_Y \end{bmatrix} \right)
\]
depends only on the class of $\delta$ in $\Ext^1_{\mc D(A_2)}(X\to Y,U\to V)$ and
\[
(d_U,f,d_V)\otimes (d_X,g,d_Y):=\left(\prod_{i=0}^{\infty}\left|\Ext^{-i}((X,d_X),(U,d_U))\right|^{(-1)^{i+1}}\right) \sum_{[\delta]}T(\delta)\nonumber
\]
where the sum is over $[\delta]\in\Ext^1_{\mc D(A_2)}(X\to Y,U\to V)\cong\Ext^1((X,d_X),(U,d_U))$.
\end{lemma}

\begin{proof}
A morphism $h\in\Hom^0{\mc D(A_2)}(X\to Y,U\to V)$ is given by a triple $(h_{11},h_{12},h_{22})$ with $h_{11}:X\to U$ and $h_{22}:Y\to V$ of degree 0 and $h_{12}:X\to V$ of degree -1.
Then
\begin{align*}
(\delta+dh)_{11}&= \delta_{11}+d_Uh_{11}-h_{11}d_X \\
(\delta+dh)_{12}&= \delta_{12}+gh_{11}-h_{22}f+d_Vh_{12}+h_{12}d_X \\
(\delta+dh)_{22}&= \delta_{22}+d_Vh_{22}-h_{22}d_Y \\
\end{align*}
and $T(\delta+dh)$ is equivalent to $T(\delta)$ via
\[
\begin{bmatrix} 1 & h_{11} \\ 0 & 1 \end{bmatrix},\qquad \begin{bmatrix} 1 & h_{22} \\ 0 & 1 \end{bmatrix},\qquad \begin{bmatrix} 0 & h_{12} \\ 0 & 0 \end{bmatrix}
\]
where the first and second block matrices are grading and filtration preserving automorphisms of $U\oplus X$ and $V\oplus Y$, respectively, and the third is a homotopy between the quasi-equivalences $U\oplus X\to V\oplus Y$.

For the second claim note that the projection 
\[
\Hom^\bullet_{\mc D(A_2)}(X\to Y,U\to V)\cong\Hom^\bullet((X,d_X),(U,d_U))
\]
is a quasi-isomorphism since $f$ and $g$ are quasi-isomorphism and thus the two representation are contained in a full subcategory of $\mc D(A_2)$ which is quasi-equivalent to the category of chain complexes.
In general, if $C$ is a finite-dimensional chain complex over $\mathbf k$ with cohomology $H(C)$ then
\begin{equation}\label{mult_neg_euler}
|\omega|\prod_{i=0}^{\infty}\left|C^{-i}\right|^{(-1)^{i+1}}=\prod_{i=0}^{\infty}\left|H^{-i}(C)\right|^{(-1)^{i+1}}
\end{equation} 
for any class $\omega\in H^1(C)$.
Combining these two fact we get
\begin{align*}
\prod_{i=0}^{\infty}&\left(\left|\Hom^{-i}(X,U)\right|\left|\Hom^{-i}(Y,V)\right|\left|\Hom^{-i-1}(X,V)\right|\right)^{(-1)^{i+1}}\\
&=\prod_{i=0}^{\infty}\left|\Hom_{\mc D(A_2)}^{-i}(X\to Y,U\to V)\right|^{(-1)^{i+1}}\\
&=|d\Hom_{\mc D(A_2)}^0|^{-1}\prod_{i=0}^{\infty}\left|\Ext_{\mc D(A_2)}^{-i}(X\to Y,U\to V)\right|^{(-1)^{i+1}}\\
&=|d\Hom_{\mc D(A_2)}^0|^{-1}\prod_{i=0}^{\infty}\left|\Ext^{-i}((X,d_X),(Y,d_Y))\right|^{(-1)^{i+1}}
\end{align*}
and 
\[
\sum_\delta T(\delta)=|d\Hom_{\mc D(A_2)}^0|\sum_{[\delta]}T(\delta)
\]
where the first sum is over cocycles and the second over cohomology classes, which together show the claimed formula.
\end{proof}

\begin{lemma}
The monoidal product defined above is bifunctorial, i.e. 
\begin{equation}\label{hv_xchange}
(\beta\otimes \delta)\circ (\alpha\otimes \gamma)=(\beta\circ\alpha)\otimes (\delta\circ\gamma)
\end{equation}
for $\alpha:U\to V$, $\beta:V\to W$, $\gamma:X\to Y$, $\delta:Y\to Z$, and
\begin{equation*}
1_V\otimes 1_W=1_{V\otimes W}.
\end{equation*}
\end{lemma}

\begin{proof}
The product on the left-hand side of \eqref{hv_xchange} is a weighted sum over certain 9-tuples of maps 
\begin{equation*}
\left(a_1,a_2,b_{11},b_{12},b_{22},b_{23},b_{33},b_{34},b_{44}\right)
\end{equation*}
with $|a_i|=|b_{i,i+1}|=0$, $|b_{ii}|=1$, fitting into in a diagram 
\begin{equation}\label{comp_diag1}
\begin{tikzcd}
X \arrow{r}{g_1}\arrow{d}{b_{11}} \arrow{dr}{b_{12}} & Y \arrow{d}{b_{22}}\arrow{dr}{b_{23}}\arrow{r}{a_1} & Y \arrow{d}{b_{33}}\arrow{dr}{b_{34}}\arrow{r}{f_1} & Z \arrow{d}{b_{44}} \\
U  \arrow{r}[swap]{g_2} & V \arrow{r}[swap]{a_2} & V \arrow{r}[swap]{f_2} & W 
\end{tikzcd}
\end{equation}
where horizontal and vertical arrows are chain maps and squares commute up to homotopies given by diagonal arrows.
Here, the differentials on $U$ and the first copy of $V$ and $g_2$ come from $\alpha$, and similarly the other differentials and $f_2,g_1,f_1$ from $\beta,\gamma,\delta$ respectively.
The product on the right-hand side of \eqref{hv_xchange} is a weighted sum over certain 5-tuples of maps
\begin{equation*}
\left(a_1,a_2,c_{11},c_{12},c_{22}\right)
\end{equation*} 
with $|a_i|=|c_{12}|=0$, $|c_{ii}|=1$, fitting into in a diagram
\begin{equation}\label{comp_diag2}
\begin{tikzcd}
X \arrow{d}{c_{11}} \arrow{drrr}{c_{12}}\arrow{r}{g_1} & Y \arrow{r}{a_1} & Y \arrow{r}{f_1} & Z\arrow{d}{c_{22}} \\
U  \arrow{r}[swap]{g_2} & V \arrow{r}[swap]{a_2} & V \arrow{r}[swap]{f_2} & W 
\end{tikzcd}
\end{equation}
where horizontal and vertical arrows are chain maps and the square commute up to homotopy given by $c_{12}$.
The condition on the $c_{ij}$, given $a_1,a_2$, is precisely that they give a closed morphism of degree one 
\begin{equation*}
(c_{11},c_{12},c_{22})\in\Hom^1_{\mc D(A_2)}(X\to Z,U\to W)
\end{equation*}
in the dg-category of representations of the $A_2$ quiver ($\bullet\to\bullet$) defined in Appendix~\ref{sec_qrep}. 
Similar, the condition on the $b_{ij}$, given $a_1,a_2$, is that they given a closed morphism of degree one
\begin{equation*}
(b_{ij})_{ij}\in\Hom^1_{\mc D(A_4)}(X\to Y\to Y\to Z,U\to V\to V\to W)
\end{equation*}
in the dg-category of representations of the $A_4$ quiver ($\bullet\to\bullet\to\bullet\to\bullet$).
Furthermore, let
\begin{gather*}
r_1 := \prod_{i=1}^{\infty}\left(\left|\Hom^{-i}_{< 0}(V,V)\right|\left|\Hom^{-i}_{< 0}(Y,Y)\right|\right)^{(-1)^i} \\
r_2 := \prod_{i=0}^{\infty}\left|\Hom^{-i}_{\mc D(A_4)}(X\to Y\to Y\to Z,U\to V\to V\to W)\right|^{(-1)^{i+1}} \\
r_3 := \prod_{i=0}^{\infty}\left|\Hom^{-i}_{\mc D(A_2)}(X\to Z,U\to W)\right|^{(-1)^{i+1}} 
\end{gather*}
then the left-hand side of \eqref{hv_xchange} is 
\begin{equation}\label{lhs_detail}
r_1r_2\sum_{a_1,a_2}\sum_b\left(\begin{bmatrix} d_U & b_{11} \\ 0 & d_X \end{bmatrix},\begin{bmatrix} f_2a_2g_2 & f_2a_2b_{12}+f_2b_{23}g_1+b_{34}a_1g_1 \\ 0 & f_1a_1g_1 \end{bmatrix},\begin{bmatrix} d_W & b_{44} \\ 0 & d_Z\end{bmatrix} \right)
\end{equation}
and the right-hand side is
\begin{equation}\label{rhs_detail}
r_1r_3\sum_{a_1,a_2}\sum_c\left(\begin{bmatrix} d_U & c_{11} \\ 0 & d_X \end{bmatrix},\begin{bmatrix} f_2a_2g_2 & c_{12} \\ 0 & f_1a_1g_1 \end{bmatrix},\begin{bmatrix} d_W & c_{22} \\ 0 & d_Z\end{bmatrix} \right).
\end{equation}

We claim that the equivalence classes of the morphisms (triples) which appear as summands in \eqref{lhs_detail} and \eqref{rhs_detail}, which we denote by $T(b)$ and $T(c)$, depend only on the classes of $[b]\in\Ext^1_{\mc D(A_4)}(X\to Y\to Y\to Z,U\to V\to V\to W)$ and $[c]\in\mathrm{Ext}^1_{\mc D(A_2)}(X\to Z,U\to W)$, respectively (as well as other maps and differentials involved).
This is very similar to what was proven in Lemma~\ref{lem_monoidal_ext}, in particular in the case of $c$.
In the case of $b$ one has to use the homotopy given by 
\[
f_2a_2h_{12}+f_2h_{23}g_1+h_{34}a_1g_1
\]
instead of just $h_{12}$.

Next, note that all horizontal arrows in \eqref{comp_diag1} (which are the same as those in \eqref{comp_diag2}) are quasi-isomorphisms, and thus there are quasi-isomorphisms of $\Hom$-complexes
\begin{align*}
\Hom_{\mc D(A_2)}(X\to Z,U\to W)&\simeq \Hom(X,U) \\
&\simeq\Hom_{\mc D(A_4)}(X\to Y\to Y\to Z,U\to V\to V\to W). \nonumber
\end{align*}
providing an isomorphism
\[
\phi:\Ext^1_{\mc D(A_4)}(X\to Y\to Y\to Z,U\to V\to V\to W)\to\Ext^1_{\mc D(A_2)}(X\to Z,U\to W).
\]
We claim that $T(\phi([b]))=T([b])$ as morphisms $U\otimes X\to V\otimes Y$.
Indeed, let $b\in\Hom^1_{\mc D(A_4)}(X\to Y\to Y\to Z,U\to V\to V\to W)$ be closed, then
\[
(c_{11},c_{12},c_{22}):=(b_{11},f_2a_2b_{12}+f_2b_{23}g_1+b_{34}a_1g_1,b_{44})
\]
is a closed in $\Hom^1_{\mc D(A_2)}(X\to Z,U\to W)$, $\phi([b])=[c]$ since $b_{11}=c_{11}$, and $T([c])=T([b])$ is clear.

Finally, note that $r_2|[b]|=r_3|[c]|$ using the quasi-isomorphism above and \eqref{mult_neg_euler}.
Putting everything together, we have
\begin{align*}
(\beta\otimes \delta)\circ (\alpha\otimes \gamma)&=r_1r_2\sum_{a_1,a_2}\sum_bT(b) \\
&=r_1r_2\sum_{a_1,a_2}\sum_{[b]}|[b]|T([b]) \\
&=r_1\sum_{a_1,a_2}\sum_{[b]}r_2|[b]|T(\phi([b])) \\
&=r_1\sum_{a_1,a_2}\sum_{[c]}r_3|[c]|T([c]) \\
&=r_1r_3\sum_{a_1,a_2}\sum_cT(c) =(\beta\circ\alpha)\otimes (\delta\circ\gamma).
\end{align*}

The second part, $1_{V\otimes W}=1_V\otimes 1_W$, is proven using the quasi-isomorphism
\begin{equation*}
\Hom(W,V)\simeq \Hom_{\mc D(A_2)}(W\xrightarrow{1_W} W,V\xrightarrow{1_V} V)
\end{equation*}
where the complex on the left appears in the definition of $1_{V\otimes W}$ and the complex on the right appears in the product $1_V\otimes 1_W$.
The identity also follows from \eqref{monoidal_prod_unit_left} or \eqref{monoidal_prod_unit_right} below.
\end{proof}

\subsection{Dualities}
\label{subsec_h_dualities}

Given a triple $(d_V,f,d_W)$ there are two natural duals to take:
\begin{equation}\label{Hdual_1}
D(d_V,f,d_W):=(d_W,f',d_V), \qquad f' \text{ inverts } f \text{ up to homotopy }
\end{equation}
and
\begin{equation}\label{Hdual_2}
(d_V,f,d_W)^\vee:=(d_W^\vee,f^\vee,d_V^\vee)
\end{equation}
induced by vector space duality.
In this subsection we show that these two operations induce contravariant autoequivalences from $\mc H$ to itself which are compatible with the monoidal product.

We first discuss the functor $D:\mc H\to\mc H$ which is the identity on objects and acts on morphisms by \eqref{Hdual_1}.
Contravariance with respect to composition in $\mc H$ uses only the fact that $b\mapsto b^{-1}$ is a bijection from filtration preserving isomorphisms of chain complexes $(V,d_V)\to (V,d_V')$ to filtration preserving isomorphisms of chain complexes $(V,d_V')\to (V,d_V)$.
It is also clear that $D$ preserves identities.

\begin{prop}
The functor $D$ is covariant with respect to the monoidal product.
\end{prop}

\begin{proof}
Let $X,Y,U,V,d_X,d_Y,d_U,d_V,f,g,\delta_{11},\delta_{12},\delta_{22}$ be as in the definition of the monoidal product, thus forming a diagram of chain maps
\begin{equation*}
\begin{tikzcd}
X \arrow{r}{g}\arrow{d}{\delta_{11}} \arrow{dr}{\delta_{12}} & Y \arrow{d}{\delta_{22}} \\
U  \arrow{r}[swap]{f} & V  
\end{tikzcd}
\end{equation*}
commuting up homotopy given by $\delta_{12}$.
We can interpret the triple $(g,\delta_{12},f)$ as a closed morphism in $\Hom^0_{\mc D(A_2)}((X\to U),(Y\to V))$ and it has an inverse up to homotopy denoted $(g',\delta_{12}',f')\in\Hom^0_{\mc D(A_2)}((Y\to V),(X\to U))$ and forming a diagram
\begin{equation*}
\begin{tikzcd}
Y \arrow{r}{g'}\arrow{d}{\delta_{22}} \arrow{dr}{\delta_{12}'} & X \arrow{d}{\delta_{11}} \\
V  \arrow{r}[swap]{f'} & U  
\end{tikzcd}
\end{equation*}
with the properties of the previous one.
Moreover, the map which sends $(\delta_{11},\delta_{12},\delta_{22})$ to $(\delta_{22},\delta_{12}',\delta_{11})$ induces an isomorphism
\[
\Ext_{\mc D(A_2)}^1((X\to Y),(U\to V))\to \Ext_{\mc D(A_2)}^1((Y\to X),(V\to U))
\]
compatible with the identification of both spaces with $\Ext^1((X,d_X),(U,d_U))$.
Setting
\[
r:=\left(\prod_{i=0}^{\infty}\left|\Ext^{-i}((X,d_X),(U,d_U))\right|^{(-1)^{i+1}}\right)=\left(\prod_{i=0}^{\infty}\left|\Ext^{-i}((Y,d_Y),(V,d_V))\right|^{(-1)^{i+1}}\right)
\]
and using Lemma~\ref{lem_monoidal_ext} and the notation $T(\delta)$ in its proof we compute
\begin{align*}
D((d_U,f,d_V)\otimes (d_X,g,d_Y))&=r\sum_{[\delta]}D(T(\delta))\\
&=r\sum_{[\delta]}T(\delta')\\
&=r\sum_{[\delta']}T(\delta')\\
&=D(d_U,f,d_V)\otimes D(d_X,g,d_Y)
\end{align*}
where the sums are over $[\delta]\in \Ext_{\mc D(A_2)}^1((X\to Y),(U\to V))$ and $[\delta']\in\Ext_{\mc D(A_2)}^1((Y\to X),(V\to U))$ respectively.
\end{proof}

We turn to the second duality, \eqref{Hdual_2}.
The functor acts on objects by $V\to V^\vee$, sending a graded vector space to its dual in the graded sense and with the dual complete flag given by
\[
F_iV^\vee:=\mathrm{Ann}(F_{\dim V-i}V)
\]
where $\mathrm{Ann}(W)\subset V^\vee$ denotes the annihilator of $W\subset V$.
We also write this as $\mathrm{Ann}_V(W)$ if the choice of ambient space needs to be emphasized.

The assignment $V\mapsto V^\vee$ on objects in $\mc H$ is contravariant with respect to the monoidal product of objects, i.e.
\[
(V\otimes W)^\vee\cong W^\vee\otimes V^\vee
\]
as graded spaces with flag, where the identification comes from the usual isomorphism $(V\oplus W)^\vee\cong W^\vee\oplus V^\vee$.
Indeed, on the left hand side we get the flag
\[
\ldots\subset \mathrm{Ann}_{V\oplus W}(V\oplus F_{\dim W-j}W)\subset\ldots\subset \mathrm{Ann}_{V\oplus W}(F_{\dim V-i}V)\subset\ldots
\]
which under the identification above is the same as
\[
\ldots\subset \mathrm{Ann}_W(F_{\dim W-j}W)\subset\ldots\subset \mathrm{Ann}_V(F_{\dim V-i}V)\oplus W^\vee\subset\ldots
\]
which we get on the right hand side.

We note that taking the dual of a map of graded vector spaces gives an isomorphism
\[
\Hom^i(V,W)\cong \Hom^i(W^\vee,V^\vee)
\]
which moreover send flags-preserving morphisms to flag-preserving ones if $V$ and $W$ are equipped with flags and $V^\vee$ and $W^\vee$ with the dual ones.
Using this fact it is straightforward to see that the functor \eqref{Hdual_2} is contravariant with respect to both the composition and the monoidal product of morphisms.

\subsection{Cones}
\label{subsec_cones}

There is an identification
\begin{equation}\label{catH_cone}
\Hom_{\mc H}(V,W) \to \Hom_{\mc H}(W[-1]\otimes V,0),\qquad (d_V,f,d_W)\mapsto \left(\begin{bmatrix} -d_W & f \\ 0 & d_V \end{bmatrix},0,0\right).
\end{equation}
using the fact that the cone over a quasi-isomorphism is an acyclic complex.
The isomorphism \eqref{catH_cone} can alternatively be written in terms of horizontal and vertical composition as follows.
Let
\[
\beta_W:=(q-1)^{-\dim W}\left(\prod_{i=0}^{\infty}\left|\Hom^{-i}_{<0}(W,W)\right|^{(-1)^{i+1}}\right)\sum_d\left(\begin{bmatrix} -d & 1 \\ 0 & d\end{bmatrix},0,0\right)
\]
where $d$ ranges over flag preserving differentials on $W$.
Note the similarity with \eqref{h_id}.

\begin{lemma}
\label{lem_Hcone}
\[
\left(\begin{bmatrix} -d_W & f \\ 0 & d_V \end{bmatrix},0,0\right)=\beta_W\circ\left(1_{W[-1]}\otimes (d_V,f,d_W)\right)
\]
\end{lemma}

\begin{proof}
Recall from \eqref{h_id} that
\[
1_{W[-1]}=(q-1)^{-\dim W}\left(\prod_{i=0}^{\infty}\left|\Hom^{-i}_{<0}(W,W)\right|^{(-1)^{i+1}}\right)\sum_d(d,1,d)
\]
where the sum ranges over differentials $d:W\to W[1]$.
We compute
\begin{align*}
(d,1,d)&\otimes (d_V,f,d_W)= \\
&\left(\prod_{i=0}^{\infty}\left(\left|\Hom^{-i}(V,W[-1])\right|\left|\Hom^{-i}(W,W[-1])\right|\left|\Hom^{-i-1}(V,W[-1])\right|\right)^{(-1)^{i+1}}\right) \\
&\cdot\sum_{\delta}\left(\begin{bmatrix} d & \delta_{11} \\ 0 & d_V \end{bmatrix},\begin{bmatrix} 1 & \delta_{12} \\ 0 & f \end{bmatrix},\begin{bmatrix} d & \delta_{22} \\ 0 & d_W \end{bmatrix} \right)
\end{align*}
where $\delta_{12}\in\Hom^0(V,W[-1])$ is arbitrary, $\delta_{22}\in\Hom^1(W,W[-1])$ satisfies $d\delta_{22}+\delta_{22}d_W=0$, and $\delta_{11}=d_W\delta_{12}+\delta_{22}f-\delta_{12}d_V$ in $\Hom^1(V,W[-1])$.
This simplifies since
\[
\prod_{i=0}^{\infty}\left(\left|\Hom^{-i}(V,W[-1])\right|\left|\Hom^{-i-1}(V,W[-1])\right|\right)^{(-1)^{i+1}}=\left|\Hom^0(V,W[-1])\right|^{-1}.
\]
and a change of basis eliminates the $\delta_{12}$ term, giving
\[
(d,1,d)\otimes (d_V,f,d_W)=\left(\prod_{i=0}^{\infty}\left|\Hom^{-i}(W,W[-1])\right|^{(-1)^{i+1}}\right) \sum_{\delta_{22}}\left(\begin{bmatrix} d & \delta_{22}f \\ 0 & d_V \end{bmatrix},\begin{bmatrix} 1 & 0 \\ 0 & f \end{bmatrix},\begin{bmatrix} d & \delta_{22} \\ 0 & d_W \end{bmatrix} \right).
\]

Next we compute the horizontal product with $\beta_W$, where we rename the $d$ in $\beta_{W}$ to $\varepsilon$ to prevent a clash of notation.
Non-zero terms come from invertible upper-triangular matrices $B=\begin{bmatrix} b_{11} & b_{12} \\ 0 & b_{22} \end{bmatrix}$ such that
\[
\begin{bmatrix} b_{11} & b_{12} \\ 0 & b_{22} \end{bmatrix} \begin{bmatrix} d & \delta_{22} \\ 0 & d_W \end{bmatrix}=\begin{bmatrix} -\varepsilon & 1 \\ 0 & \varepsilon \end{bmatrix}\begin{bmatrix} b_{11} & b_{12} \\ 0 & b_{22} \end{bmatrix}
\]
but this implies
\[
\begin{bmatrix} c_{11} & c_{12} \\ 0 & 1 \end{bmatrix}\begin{bmatrix} d & \delta_{22}f \\ 0 & d_V \end{bmatrix}=\begin{bmatrix} -d_W & f \\ 0 & d_V \end{bmatrix}\begin{bmatrix} c_{11} & c_{12} \\ 0 & 1 \end{bmatrix}
\]
with $c_{11}:=b_{22}^{-1}b_{11}$ and $c_{12}:=b_{22}^{-1}b_{12}f$.
Hence

\begin{align*}
\left({\begin{bmatrix} -d_W & f \\ 0 & d_V \end{bmatrix}},0,0 \right)= 
&\left(\prod_{i=0}^{\infty}\left|\Hom^{-i}_{<0}(W[-1]\oplus W,W[-1]\oplus W)\right|^{(-1)^i}\right)\left({\begin{bmatrix} -\varepsilon & 1 \\ 0 & \varepsilon, \end{bmatrix}}, 0, 0 \right) \\ &\circ
\sum_{d,\delta,\varepsilon} \left({\begin{bmatrix} d & \delta_{22}f \\ 0 & d_V \end{bmatrix}},{\begin{bmatrix} 1 & 0 \\ 0 & f \end{bmatrix}},{\begin{bmatrix} d & \delta_{22} \\ 0 & d_W \end{bmatrix}} \right) 
\end{align*}
but
\[
\Hom^{-i}_{<0}(W[-1]\oplus W,W[-1]\oplus W)=\Hom^{-i}_{<0}(W,W)\oplus\Hom^{-i}_{<0}(W,W)\oplus\Hom^{-i-1}(W,W)
\]
so several terms cancel giving the desired formula.
\end{proof}

Using Proposition~\ref{prop_dclass} we can describe  the sets $B(V,W)$ of equivalence classes of triples $(d_V,f,d_W)$ which, by definition, provide a basis for $\Hom_{\mc H}(V,W)$, more explicitly .
First, in the special case $W=0$, there is a one-to-one correspondence between $B(V,0)$ and the set of rulings of $\Phi^{-1}V$.
For the general case we can use the bijection $B(V,W)\cong B(W[-1]\otimes V,0)$ from \eqref{catH_cone} to obtain:  

\begin{prop}
Suppose $V_1,V_2\in\mathrm{Ob}(\mc H)$ and let $X_i=\Phi^{-1}(V_i)$, $i=1,2$, be the corresponding totally ordered graded sets.
The set $B(V_1,V_2)$ is canonically identified with the set of quintuples $(D_1,\delta_1,D_2,\delta_2,\sigma)$ where $(D_i,\delta_i)$ is a partial ruling of $X_i$ and 
\[
\sigma:X_1\setminus (D_1\cup\delta_1(D_1))\to X_2\setminus (D_2\cup\delta_2(D_2))
\]
is a grading-preserving bijection.
\end{prop}

\section{The category $\mc S$}
\label{sec_S}

The goal of this section is to obtain explicit bases of $\Hom_{\mc S}(X,Y)$.
In the first subsection we discuss how to reduce the problem to the case $Y=\emptyset$.
In Subsection~\ref{subsec_gen} we provide a small spanning set for $\Hom_{\mc S}(X,\emptyset)$.
In Subsection~\ref{subsec_rulings} we define rulings of tangles with right boundary only, which give a way to construct a natural basis of $\Hom_{\mc S}(X,\emptyset)$.
The final subsection describes two dualities of the category $\mc S$ which are used to simplify proofs later.

\subsection{Bending tangles}
\label{subsec_bend}

Given a graded Legendrian $L$ and $n\in\ZZ$ denote by $L[n]$ the same underlying Legendrian curve but with the grading $k$ on each strand replaced by $k-n$. This induces an autoequivalence of $\mc S$.
There is a canonical isomorphism (even without imposing skein relations)
\begin{equation}\label{tangle_bend_iso}
\Hom_{\mc S}(X,Y) \cong \Hom_{\mc S}(Y[-1]\otimes X,\emptyset)
\end{equation}
which takes a tangle, $L$, and reattaches the left end at the right boundary below the right end of $L$.
This operation looks simpler when viewed under Lagrangian projection and assuming that the boundary $X$ and $Y[-1]$ are placed at an offset in the $y$ direction.
The argument to show this operation is a bijection on Legendrian isotopy classes is then the usual straightening of a cap-cup.

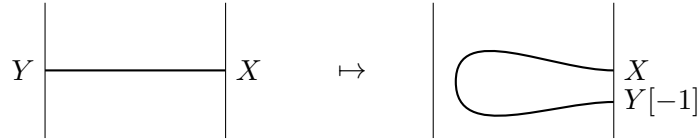
\[
\begin{tikzpicture}[baseline=-\dimexpr\fontdimen22\textfont2\relax,scale=.6]
\draw (-2,-1.5) to (-2,1.5);
\draw (2,-1.5) to (2,1.5);
\draw[thick] (-2,0) to (2,0);
\node[left] at (-2,0) {$Y$};
\node[right] at (2,0) {$X$};
\end{tikzpicture}
\qquad\mapsto\qquad
\begin{tikzpicture}[baseline=-\dimexpr\fontdimen22\textfont2\relax,scale=.6]
\draw (-2,-1.5) to (-2,1.5);
\draw (2,-1.5) to (2,1.5);
\draw[thick] (2,0) to [out=180,in=90] (-1.5,-.25) to [out=-90,in=180] (2,-.7);
\node[right] at (2,-.7) {$Y[-1]$};
\node[right] at (2,0) {$X$};
\end{tikzpicture}
\]
 
Algebraically, we can express \eqref{tangle_bend_iso} as $L\mapsto \beta_Y\circ (1_{Y[-1]}\otimes L)$ where $\beta_Y\in\Hom_{\mc S}(Y[-1]\otimes Y,\emptyset)$ is the tangle in Figure~\ref{fig_tangle_bend} and is a union of several left-cusps, their number and grading corresponding to $Y$.
Given \eqref{tangle_bend_iso}, we reduce the problem of determining $\Hom_{\mc S}(X,Y)$ to the special case where $Y=\emptyset$.

\begin{figure}[ht]
\centering
\begin{tikzpicture}[scale=.6]
\draw[white,line width=1mm] (0,.5) to [out=180,in=0] (-3,-1) to [out=0,in=180] (0,-2.5);
\draw[thick] (0,.5) to [out=180,in=0] (-3,-1) to [out=0,in=180] (0,-2.5);
\draw[white,line width=1mm] (0,2.2) to [out=180,in=0] (-3,.7) to [out=0,in=180] (0,-.8);
\draw[thick] (0,2.2) to [out=180,in=0] (-3,.7) to [out=0,in=180] (0,-.8);
\draw[white,line width=1mm] (0,2.5) to [out=180,in=0] (-3,1) to [out=0,in=180] (0,-.5);
\draw[thick] (0,2.5) to [out=180,in=0] (-3,1) to [out=0,in=180] (0,-.5);
\draw[thick] (0,-3) to (0,3);
\node[right] at (0,1.5) {$Y$};
\node[right] at (0,-1.5) {$Y[-1]$};
\node at (-.5,1.5) {$\vdots$};
\node at (-.5,-1.5) {$\vdots$};
\node at (-2.7,0) {$\vdots$};
\end{tikzpicture}
\caption{The tangle $\beta_Y$.}
\label{fig_tangle_bend}
\end{figure}
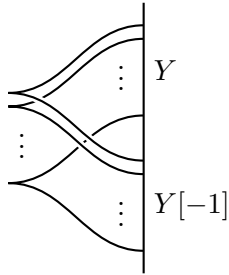

\subsection{Generating the skein}
\label{subsec_gen}

The following gives an upper bound on $\Hom_{\mc S}(X,\emptyset)$ which will be used in the next subsection to determine a basis.

\begin{lemma}\label{lem_tanglegen}
The module $\Hom_{\mc S}(X,\emptyset)$ is generated by tangles which are horizontal compositions of tangles as in Figure~\ref{fig_addcusp} with $n_1=n_2=0$.
\end{lemma}

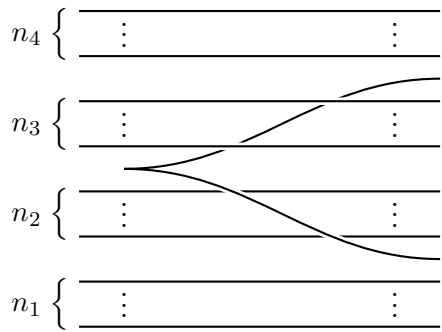
\begin{figure}[ht]
\centering
\begin{tikzpicture}[scale=.6]
\draw[thick] (-4,-3.5) to (4,-3.5);
\node at (-3,-2.85) {\vdots};
\node at (3,-2.85) {\vdots};
\draw[thick] (-4,-2.5) to (4,-2.5);
\draw[white,line width=1mm] (-4,-1.5) to (4,-1.5);
\draw[thick] (-4,-1.5) to (4,-1.5);
\node at (-3,-.85) {\vdots};
\node at (3,-.85) {\vdots};
\draw[white,line width=1mm] (-4,-.5) to (4,-.5);
\draw[thick] (-4,-.5) to (4,-.5);

\draw[white,line width=1mm] (4,2) to [out=180,in=0] (-3,0) to [out=0,in=180] (4,-2);
\draw[thick] (4,2) to [out=180,in=0] (-3,0) to [out=0,in=180] (4,-2);

\draw[thick] (-4,3.5) to (4,3.5);
\node at (-3,3.15) {\vdots};
\node at (3,3.15) {\vdots};
\draw[thick] (-4,2.5) to (4,2.5);
\draw[white,line width=1mm] (-4,1.5) to (4,1.5);
\draw[thick] (-4,1.5) to (4,1.5);
\node at (-3,1.15) {\vdots};
\node at (3,1.15) {\vdots};
\draw[white,line width=1mm] (-4,.5) to (4,.5);
\draw[thick] (-4,.5) to (4,.5);

\node[left] at (-4,3) {$n_4\;\Big{\{}$};
\node[left] at (-4,1) {$n_3\;\Big{\{}$};
\node[left] at (-4,-1) {$n_2\;\Big{\{}$};
\node[left] at (-4,-3) {$n_1\;\Big{\{}$};

\end{tikzpicture}
\caption{A tangle with a single left cusp and $n_2+n_3$ crossings.}
\label{fig_addcusp}
\end{figure}

\begin{proof}
We will arrive at the statement of the lemma via a route of successively stronger claims.

\begin{enumerate}
\item
\textit{Claim: $\Hom_{\mc S}(X,\emptyset)$ is generated by tangles without right cusps.}

This follows from the proof in~\cite[Section 3]{rutherford06}.
Roughly, the idea is to consider the left-most right cusp as part of a tangle mirror to the one in Figure~\ref{fig_addcusp} and perform a case-by-case analysis based on the basic tangle immediately to the left.
In all cases one can, using planar isotopy, Reidemeister moves, and skein relations, reduce either the number of crossings or simplify the part of the tangle to the left of the cusp, and so the claim follows by a nested induction argument.

\item
\textit{Claim: $\Hom_{\mc S}(X,\emptyset)$ is generated by tangles which are compositions of tangles as in Figure~\ref{fig_addcusp}.
Moreover, a given tangle $L$ without right cusps can be written as a linear combination of such tangles which have the same number or less crossings than $L$.}

Using the first claim we may restrict to tangles without right cusps.
One considers the right-most left cusp as part of tangle as in Figure~\ref{fig_addcusp} and the basic tangle, necessarily a crossing, immediately to the right of it.
As before, the case-by-case analysis in~\cite{rutherford06} shows one can reduce the number of crossings to the right of the cusp without increasing the total number of crossings.

\item
\textit{Claim: $\Hom_{\mc S}(X,\emptyset)$ is generated by tangles which are a composition $L_n\circ\cdots\circ L_1$ of tangles $L_1,\ldots,L_n$ as in Figure~\ref{fig_addcusp} and if $k\in\{1,\ldots,n\}$ such that $L_k$ contains the left cusp whose bottom strand connects to the lowest point in $X$ then $n_2=0$ for $L_k$ and $n_3=0$ for $L_1,\ldots L_{k-1}$. }

Let $L$ be a tangle without right cusps. 
We show by induction on the number of crossings in $L$ that it may be written as a linear combination of tangles as in the statement of the claim.
By the previous claim we may assume that $L$ is a composition $L_n\circ\cdots\circ L_1$ of tangles $L_1,\ldots,L_n$ as in Figure~\ref{fig_addcusp}. 
Let $k\in\{1,\ldots,n\}$ such that $L_k$ contains the left cusp whose bottom strand connects to the lowest point in $X$.
The skein relation \eqref{gs_1f} allows us to decrease $n_2$ for $L_k$, should it be positive, modulo terms which have less crossings and are thus taken care of by induction, until $n_2=0$.
Similarly we may decrease $n_3$ for $L_1,\ldots,L_{k-1}$ using \eqref{gs_1f} until reaching a tangle as in the statement of the claim.
\end{enumerate}

Finally, suppose $L=L_n\circ\cdots\circ L_1$ is as in the third claim, then we can use a Legendrian isotopy to move the left-cusp in $L_k$ to the right past all the left cusps in $L_1,\ldots,L_{k-1}$ giving a tangle which factors into a tangle as in Figure~\ref{fig_addcusp} with $n_1=n_2=0$ and a tangle with two boundary points less.
The lemma follows by induction on $|X|$.
\end{proof}

\subsection{Rulings}
\label{subsec_rulings}

Counting \textit{rulings} of Legendrian links provide a way of extracting invariants under isotopy and skein moves.
The technique, originally due to Chekanov--Pushkar~\cite{cp_4conj} based on earlier ideas of Eliashberg~\cite{eliashberg87}, may be used to find a basis of the skein of tangles, i.e. of $\Hom_{\mc S}(X,Y)$ for any $X,Y\in \mathrm{Ob}(\mc S)$.
This will lead to an alternative definition of $\Phi$ and the proof of the main theorem.
Rulings of Legendrian tangles where also considered in~\cite{taosu}, which generalized several results from Legendrian knot theory to tangles.

Let $L$ be a tangle with $\partial_0L=\emptyset$.
Suppose also that the front projection of $L$ is generic in the sense that all singularities are cusps and simple crossings and project to distinct points on the $x$-axis.
A (graded, normal) \textbf{ruling} $\rho$ of $L$ is given by a collection of closed intervals $I_i\subset [0,1]$ and piecewise smooth functions $\alpha_i,\beta_i:I_i\to\RR$, $i=1,\ldots,n$ such that 
\begin{enumerate}[1)]
\item
The graphs of $\alpha_i,\beta_i$ are contained in the front projection of $L$.
\item
For $x\in\partial I_i$, $x\neq 1$, the endpoints $(x,\alpha_i(x))$, $(x,\beta_i(x))$ are the same cusp of $L$.
\item
$\alpha_i(x)\geq \beta_i(x)$, $x\in I_i$, with equality iff $x\in\partial I_i$, $x\neq 1$.
\item
The graphs of $\alpha_1,\ldots,\alpha_n,\beta_1,\ldots,\beta_n$ together cover all of $L$ and intersect only at crossings and cusps.
\item
At a crossing, a path may switch from one strand to another, but only if the strands have the same grading degree, and furthermore the following condition is satisfied.
Let the paths meeting at the crossing be one of $\alpha_i,\beta_i$ and one of $\alpha_j,\beta_j$, then for $x$ just to the left (equivalently right) of the $x$-coordinate of the crossing, the two intervals $[\alpha_i(x),\beta_i(x)]$, $[\alpha_j(x),\beta_j(x)]$ must either be disjoint or one must entirely contain the other (see top row in Figure~\ref{fig_switches_departures}). 
A crossing where a pair of paths jumps strands is called a \textbf{switch} of the ruling.
A ruling is completely determined by its set of switches.
\end{enumerate}

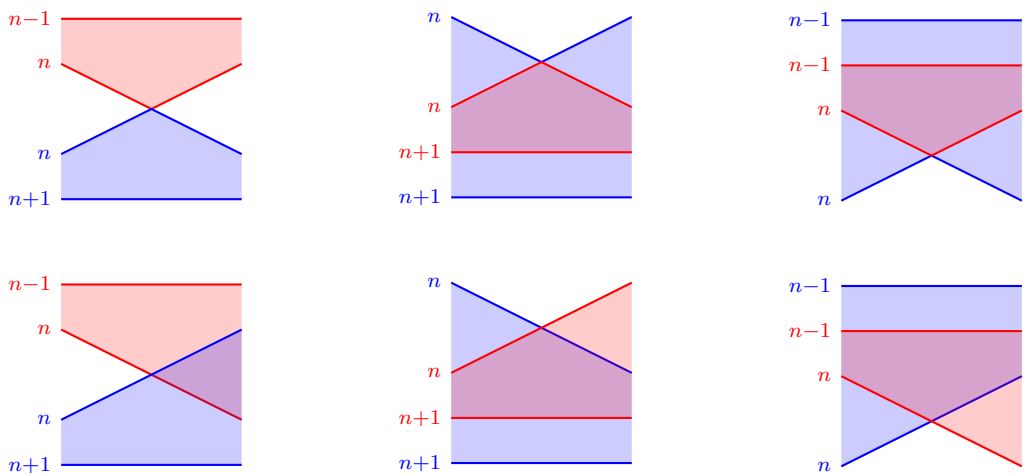
\begin{figure}[h]
\centering
\begin{subfigure}{0.25\textwidth}
\begin{tikzpicture}[baseline=-\dimexpr\fontdimen22\textfont2\relax,scale=.6]
\draw[thick,red] (-2,2) to (2,2);
\draw[thick,red] (-2,1) to (0,0) to (2,1);
\fill[opacity=.2,red] (-2,2) -- (2,2) -- (2,1) -- (0,0) -- (-2,1) -- cycle;
\draw[thick,blue] (-2,-1) to (0,0) to (2,-1);
\draw[thick,blue] (-2,-2) to (2,-2);
\fill[opacity=.2,blue] (-2,-1) -- (0,0) -- (2,-1) -- (2,-2) -- (-2,-2) -- cycle;
\node[left,red] at (-2,2) {$\scriptstyle{n-1}$};
\node[left,red] at (-2,1) {$\scriptstyle{n}$};
\node[left,blue] at (-2,-1) {$\scriptstyle{n}$};
\node[left,blue] at (-2,-2) {$\scriptstyle{n+1}$};
\end{tikzpicture}%
\end{subfigure}\hfil
\begin{subfigure}{0.25\textwidth}
\begin{tikzpicture}[baseline=-\dimexpr\fontdimen22\textfont2\relax,scale=.6]
\draw[thick,blue] (-2,3) to (0,2) to (2,3);
\draw[thick,blue] (-2,-1) to (2,-1);
\fill[opacity=.2,blue] (-2,3) -- (0,2) -- (2,3) -- (2,-1) -- (-2,-1) -- cycle;
\draw[thick,red] (-2,1) to (0,2) to (2,1);
\draw[thick,red] (-2,0) to (2,0);
\fill[opacity=.2,red] (-2,1) -- (0,2) -- (2,1) -- (2,0) -- (-2,0) -- cycle;
\node[left,blue] at (-2,3) {$\scriptstyle{n}$};
\node[left,red] at (-2,1) {$\scriptstyle{n}$};
\node[left,red] at (-2,0) {$\scriptstyle{n+1}$};
\node[left,blue] at (-2,-1) {$\scriptstyle{n+1}$};
\end{tikzpicture}%
\end{subfigure}\hfil
\begin{subfigure}{0.25\textwidth}
\begin{tikzpicture}[baseline=-\dimexpr\fontdimen22\textfont2\relax,scale=.6]
\draw[thick,blue] (-2,-3) to (0,-2) to (2,-3);
\draw[thick,blue] (-2,1) to (2,1);
\fill[opacity=.2,blue] (-2,-3) -- (0,-2) -- (2,-3) -- (2,1) -- (-2,1) -- cycle;
\draw[thick,red] (-2,-1) to (0,-2) to (2,-1);
\draw[thick,red] (-2,0) to (2,0);
\fill[opacity=.2,red] (-2,-1) -- (0,-2) -- (2,-1) -- (2,0) -- (-2,0) -- cycle;
\node[left,blue] at (-2,-3) {$\scriptstyle{n}$};
\node[left,red] at (-2,-1) {$\scriptstyle{n}$};
\node[left,red] at (-2,0) {$\scriptstyle{n-1}$};
\node[left,blue] at (-2,1) {$\scriptstyle{n-1}$};
\end{tikzpicture}%
\end{subfigure}\hfil

\medskip\medskip\medskip
\begin{subfigure}{0.25\textwidth}
\begin{tikzpicture}[baseline=-\dimexpr\fontdimen22\textfont2\relax,scale=.6]
\draw[thick,red] (-2,2) to (2,2);
\draw[thick,red] (-2,1) to (0,0) to (2,-1);
\fill[opacity=.2,red] (-2,2) -- (2,2) -- (2,-1) -- (0,0) -- (-2,1) -- cycle;
\draw[thick,blue] (-2,-1) to (0,0) to (2,1);
\draw[thick,blue] (-2,-2) to (2,-2);
\fill[opacity=.2,blue] (-2,-1) -- (0,0) -- (2,1) -- (2,-2) -- (-2,-2) -- cycle;
\node[left,red] at (-2,2) {$\scriptstyle{n-1}$};
\node[left,red] at (-2,1) {$\scriptstyle{n}$};
\node[left,blue] at (-2,-1) {$\scriptstyle{n}$};
\node[left,blue] at (-2,-2) {$\scriptstyle{n+1}$};
\end{tikzpicture}%
\end{subfigure}\hfil
\begin{subfigure}{0.25\textwidth}
\begin{tikzpicture}[baseline=-\dimexpr\fontdimen22\textfont2\relax,scale=.6]
\draw[thick,blue] (-2,3) to (0,2) to (2,1);
\draw[thick,blue] (-2,-1) to (2,-1);
\fill[opacity=.2,blue] (-2,3) -- (0,2) -- (2,1) -- (2,-1) -- (-2,-1) -- cycle;
\draw[thick,red] (-2,1) to (0,2) to (2,3);
\draw[thick,red] (-2,0) to (2,0);
\fill[opacity=.2,red] (-2,1) -- (0,2) -- (2,3) -- (2,0) -- (-2,0) -- cycle;
\node[left,blue] at (-2,3) {$\scriptstyle{n}$};
\node[left,red] at (-2,1) {$\scriptstyle{n}$};
\node[left,red] at (-2,0) {$\scriptstyle{n+1}$};
\node[left,blue] at (-2,-1) {$\scriptstyle{n+1}$};
\end{tikzpicture}%
\end{subfigure}\hfil
\begin{subfigure}{0.25\textwidth}
\begin{tikzpicture}[baseline=-\dimexpr\fontdimen22\textfont2\relax,scale=.6]
\draw[thick,blue] (-2,-3) to (0,-2) to (2,-1);
\draw[thick,blue] (-2,1) to (2,1);
\fill[opacity=.2,blue] (-2,-3) -- (0,-2) -- (2,-1) -- (2,1) -- (-2,1) -- cycle;
\draw[thick,red] (-2,-1) to (0,-2) to (2,-3);
\draw[thick,red] (-2,0) to (2,0);
\fill[opacity=.2,red] (-2,-1) -- (0,-2) -- (2,-3) -- (2,0) -- (-2,0) -- cycle;
\node[left,blue] at (-2,-3) {$\scriptstyle{n}$};
\node[left,red] at (-2,-1) {$\scriptstyle{n}$};
\node[left,red] at (-2,0) {$\scriptstyle{n-1}$};
\node[left,blue] at (-2,1) {$\scriptstyle{n-1}$};
\end{tikzpicture}%
\end{subfigure}\hfil
\caption{\textbf{Switches} (top row) and \textbf{departures} (bottom row). \textbf{Returns} are departures reflected on the vertical axis.}
\label{fig_switches_departures}
\end{figure}

Let $\mc R(L)$ denote the set of rulings of $L$. 
Any ruling of a tangle induces a ruling of the boundary of $L$ in the following way.
Recall from Subsection~\ref{subsec_flags} that a ruling of a graded subset $X\subset\RR$ is given by a subset $D\subset X$ and a bijection $\delta:D\to X\setminus D$ such that
\begin{equation}\label{set_ruling_cond_}
\delta(x)<x,\qquad \deg(\delta(x))=\deg(x)+1
\end{equation}
for all $i\in D$.
The set of such rulings will be denoted by $\mc R(X)$.
A ruling $\rho$ of a tangle $L$ with $\partial_0L=\emptyset$ induces a ruling $(D,\delta)=\partial \rho$ of $X=\partial_1L$ where $\delta(x)=y$ if there is an $i$ with $\alpha_i(1)=x$ and $\beta_i(1)=y$.

In order to get an invariant under skein moves we need to count each ruling, $\rho\in\mc R(L)$, with a certain weight, $c_\rho\in R:=\ZZ[q^{\pm},(q-1)^{-1}]$ which is a product of factors corresponding to cusps and crossings of $L$, see the table below.
(Recall \eqref{elem_tangles} for the notation of the elementary tangles $\rho_n$, $\lambda_n$, $\sigma_{m,n}$.)

\begin{center}
\begin{tabular}{c|c}
Type of cusp/crossing & Factor of $c_\rho$ \\
\hline
left cusp $\rho_n$ & $(q-1)^{-1}$ \\
right cusp $\lambda_n$ & $1$ \\ 
$\sigma_{m,n}$, $m<n$ & $q^{(-1)^{n-m}}$ \\
$\sigma_{m,n}$, $m>n$ & $1$ \\
$\sigma_{n,n}$, switch & $q-1$ \\
$\sigma_{n,n}$, departure & $1$ \\
$\sigma_{n,n}$, return & $q$
\end{tabular}
\end{center}

Define
\begin{equation*}
\nu(L):=\sum_{\rho\in\mc R(L)}c_\rho \partial\rho\in R^{\mc R(\partial_1L)}
\end{equation*}
as the total count of rulings of $L$.

\begin{lemma}
$\nu(L)$ is invariant under Legendrian isotopy.
Moreover, $\nu$ is compatible with skein relations in the sense that it induces a map 
\begin{equation*}
\nu_X:\Hom_{\mc S}(X,\emptyset)\to R^{\mc R(X)}
\end{equation*}
for any $X\in\mathrm{Ob}(\mc S)$.
\end{lemma}

\begin{proof}
On needs to check invariance under three types of moves: 1) a cusp/crossing passing over another cusp/crossing (coinciding $x$-coordinate), 2) Reidemeister moves (see Figure~\ref{fig_reidemeister}), and 3) skein moves.
This is a lengthy but straightforward case-by-case analysis, c.f. \cite{cp_4conj}, which we omit since invariance will follow instead from Proposition~\ref{prop_phi_rulings} which gives an alternative definition of $\nu_X$.
\end{proof}

\begin{figure}
\begin{align}
\begin{tikzpicture}[scale=.6]
\draw[thick] (-1.5,0) to (1.5,0);
\end{tikzpicture}%
\qquad&\longleftrightarrow\qquad
\begin{tikzpicture}[scale=.6]
\draw[thick] (-1.5,0) to [out=0,in=180] (1,1);
\draw[line width=1mm,white] (-1,1) to [out=0,in=180] (1.5,0);
\draw[thick] (-1,1) to [out=0,in=180] (1.5,0);
\draw[thick] (1,1) to [out=180,in=0] (0,1.5) to [out=180,in=0] (-1,1);
\end{tikzpicture}%
\tag{R1}\label{reid1}\\
\begin{tikzpicture}[baseline=-\dimexpr\fontdimen22\textfont2\relax,scale=.6]
\draw[thick] (-1,-1) to (1,1);
\draw[thick] (-1.5,-.3) to [out=30,in=180] (-.5,0) to [out=180,in=-30] (-1.5,.3);
\end{tikzpicture}%
\qquad&\longleftrightarrow\qquad
\begin{tikzpicture}[baseline=-\dimexpr\fontdimen22\textfont2\relax,scale=.6]
\draw[thick] (-1,-1) to (1,1);
\draw[line width=1mm,white] (-1.5,-1) to [out=30,in=180] (1,0) to [out=180,in=-30] (-1.5,1);
\draw[thick] (-1.5,-1) to [out=30,in=180] (1,0) to [out=180,in=-30] (-1.5,1);
\end{tikzpicture}%
\tag{R2}\label{reid2}\\
\begin{tikzpicture}[baseline=-\dimexpr\fontdimen22\textfont2\relax,scale=.6]
\draw[thick] (-.7,-1) to (1.3,1);
\draw[line width=1mm,white] (-1.5,.3) to (1.5,.3);
\draw[thick] (-1.5,.3) to (1.5,.3);
\draw[line width=1mm,white] (-1.3,1) to (.7,-1);
\draw[thick] (-1.3,1) to (.7,-1);
\end{tikzpicture}%
\qquad&\longleftrightarrow\qquad
\begin{tikzpicture}[baseline=-\dimexpr\fontdimen22\textfont2\relax,scale=.6]
\draw[thick] (-1.3,-1) to (.7,1);
\draw[line width=1mm,white] (-1.5,-.3) to (1.5,-.3);
\draw[thick] (-1.5,-.3) to (1.5,-.3);
\draw[line width=1mm,white] (-.7,1) to (1.3,-1);
\draw[thick] (-.7,1) to (1.3,-1);
\end{tikzpicture}%
\tag{R3}\label{reid3}
\end{align}
\caption{Legendrian Reidemeister moves}
\label{fig_reidemeister}
\end{figure}
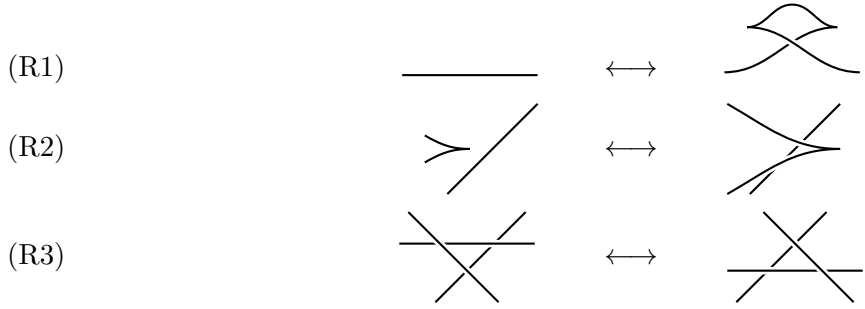

Our next goal is to show that $\nu_X$ is an isomorphism, and thus $\Hom_{\mc S}(X,\emptyset)$ has a basis given by the set of rulings of $X$.
Since $\Hom_{\mc S}(X,Y)\cong \Hom_{\mc S}(Y[-1]\otimes X,\emptyset)$, see \eqref{tangle_bend_iso}, this yields an explicit basis of all $\Hom$-spaces in $\mc S$.

\begin{prop}
\label{prop_nu_iso}
The map $\nu_X:\Hom_{\mc S}(X,\emptyset)\to R^{\mc R(X)}$ is an isomorphism for any $X\in\mathrm{Ob}(\mc S)$.
\end{prop}

\begin{proof}
We prove the statement by induction on the size $|X|$ of $X\in\mathrm{Ob}(\mc S)$.
The base case $|X|=0$ follows from the fact that $\Hom_{\mc S}(\emptyset,\emptyset)$ is generated by the empty link, a special case of Lemma~\ref{lem_tanglegen}.
Also, if $|X|$ is odd then $\Hom_{\mc S}(X,\emptyset)=0$ and $\mc R(X)$ are empty.

In the general case let $y\in X$ be the smallest element and 
\begin{equation*}
J:=\left\{ x\in X \mid \deg(x)=\deg(y)-1\right\}
\end{equation*}
then there is a bijection
\begin{equation}\label{rulings_part}
\mc R(X)\cong \bigsqcup_{x\in J}\mc R\left(X\setminus\{x,y\}\right)
\end{equation}
where a ruling of $X\setminus\{x,y\}$, extended by $\delta(x)=y$, gives a ruling of $X$.
Consider the diagram
\[
\begin{tikzcd}
\bigoplus_{x\in J}\Hom_{\mc S}\left(X\setminus\{x,y\},\emptyset\right) \arrow{d}\arrow{r}{\nu_{X\setminus\{x,y\}}} & \bigoplus_{x\in J} R^{\mc R(X\setminus\{x,y\})} \arrow{d} \\
\Hom_{\mc S}(X,\emptyset) \arrow{r}[swap]{\nu_X} & R^{\mc R(X)}
\end{tikzcd}
\]
where the left vertical arrow is given on $\Hom_{\mc S}\left(X\setminus\{x,y\},\emptyset\right)$ by composition on the right with a tangle of the form shown in Figure~\ref{fig_addcusp} with $n_1=n_2=0$ and the right vertical arrow is defined so that the diagram commutes, which is possible since the top horizontal arrow is an isomorphism by the induction hypothesis.
The left vertical map is surjective by Lemma~\ref{lem_tanglegen}.
Hence, it suffices to show that the right vertical map is an isomorphism to conclude that $\nu_X$ is an isomorphism.

Choose a total order on each of the sets $\mc R\left(X\setminus\{x,y\}\right)$, $y\in J$.
This, together with the rule that 
\begin{equation*}
\delta^{-1}(y)<\delta'^{-1}(y)\implies \delta<\delta'
\end{equation*}
and \eqref{rulings_part} determines a total order of $\mc R(X)$, thus an ordered basis on $R^{\mc R(X)}$.
We also have a corresponding basis of $\bigoplus_{x\in J} R^{\mc R(X\setminus\{x,y\})}$ of the same size.
The matrix of the right vertical arrow in the diagram with respect to these bases is block upper triangular with diagonal blocks which are scalar multiples of the identity. 
The scalars on the diagonal are of the form $q^m(q-1)^n$, so in particular units in $R$.
This follows from considering how a ruling of a tangle $L$ with $\partial_1L=X\setminus\{x,y\}$ extends to a ruling of $L$ composed with a tangle as in Figure~\ref{fig_addcusp}.
The coefficients in the diagonal blocks come from extensions of rulings without additional switches.
\end{proof}

\begin{remark}
The proof of Proposition~\ref{prop_nu_iso} shows that the tangles from Lemma~\ref{lem_tanglegen} form a basis of $\Hom_{\mc S}(X,\emptyset)$, in general different from the basis coming from $\mc R(X)$ which does not have an obvious geometric interpretation.
\end{remark}

Finally, as an example and for later purposes, we compute $\nu_{Y[-1]\otimes Y}(\beta_Y)$ where $\beta_Y$ is the tangle in Figure~\ref{fig_tangle_bend}.
To state the result, recall that $(D,\delta)$ is a \textit{partial ruling} of $Y$ if $D\subseteq Y$, $\delta:D\to Y\setminus D$ is injective and \eqref{set_ruling_cond_} holds, but we do not require that the image of $\delta$ is all of $Y\setminus D$.
Let $\mc R_{\mathrm{part}}(Y)$ denote the set of partial rulings.
Any partial ruling, $(D,\delta)$, of $Y$ gives a ruling $(\overline{D},\overline{\delta})$ of $Y[-1]\otimes Y$ as follows.
Let $n=|Y|$ and identify $Y$ (resp. $Y[-1]\otimes Y$) with $\{1,\ldots,n\}$ (resp. $\{1,\ldots,2n\}$) as ordered sets then
\begin{align*}
\overline{D}&:=D\cup (D+n)\cup \left((Y\setminus (D\cup \delta(D)))+n\right) \\
\overline{\delta}(i)&:=\begin{cases} \delta(i) & i\in D \\ \delta(i-n)+n & i\in D+n \\ i-n & i\in (Y\setminus (D\cup \delta(D))+n\end{cases}
\end{align*}
where $D+n:=\{y+n\mid y\in Y\}$ is the translated set.

\begin{lemma}
\label{lem_beta_rulings}
If $Y\in\Ob(\mc S)$ and $\beta_Y$ is the tangle in Figure~\ref{fig_tangle_bend} then
\begin{equation*}
\nu_{Y[-1]\otimes Y}(\beta_Y)=\sum_{(D,\delta)\in\mc R_{\mathrm{part}}(Y)}(q-1)^{-|Y|+|D|}q^{a(D,\delta)}(\overline{D},\overline{\delta})
\end{equation*}
where 
\[
a(D,\delta):=s(D,\Delta)-\sum_{\substack{1\leq i<j\leq |Y| \\ \deg(i)\leq\deg(j)}}(-1)^{\deg(i)-\deg(j)}
\]
and 
\[
s(D,\delta):=\left|\left\{(i,j)\mid \deg(i)=\deg(j)+1,j\in D,\delta(j)>i\text{ or }i\in\delta(D),\delta^{-1}(i)<j\right\}\right|
\]
as before.
\end{lemma}

\begin{proof}
The tangle $\beta_Y$ has $n=|Y|$ left cusps and $n(n-1)/2$ crossings where the $i$-th upper strand crosses the $j$-th lower strand, $1\leq i<j\leq n$.
Given a partial ruling $(D,\delta)$ of $Y$ define a ruling $\rho(D,\delta)$ of $\beta_Y$ such that the crossings of the form $(\delta(j),j)$, $j\in D$ are switches.
The ruling $\rho(D,\delta)$ of $\beta_Y$ restricts to the ruling $(\overline{D},\overline{\delta})$ of $Y[-1]\otimes Y$.
The \textit{Cusp Lemma} of Ng--Rutherford~\cite{ng_rutherford} implies that this defines a one-to-one correspondence between partial rulings of $Y$ and rulings of $\beta_Y$.
The main point is that the two paths of the ruling starting at a given left cusp encounter at most one switch in total.
Thus, for each left cusp one of three things occurs.
\begin{enumerate}[1)]
\item
Neither of the paths starting at the cusp encounters a switch, so the upper path ends at $Y$ while the lower path ends at $Y[-1]$.
\item 
The lower path encounters a switch, so both paths end at $Y$.
\item
The upper path encounters a switch, so both paths end at $Y[-1]$.
\end{enumerate}
To compute $c_{\rho(D,\delta)}$ we need to determine the contributions from the cusps and crossings.
By definition, the number of switches of $\rho(D,\delta)$ is $|D|$, so the cusps and switches contribute a factor $(q-1)^{-|Y|+|D|}$.
Let $i<j$ and consider the crossing where the $i$-th upper strand, which has degree $\deg(i)$, meets the $j$-th lower strand, which has degree $\deg(j)+1$.
If $\deg(i)\leq \deg(j)$, then the crossing contributes a factor $q^{-(-1)^{\deg(i)-\deg(j)}}$.
If $\deg(i)=\deg(j)+1$, then the crossing is a switch if $i=\delta(j)$, a return if there is a switch on one of the two strands and to the left of the crossing, i.e. if $j\in D$, $\delta(j)>i$ or $i\in\delta(D)$, $\delta^{-1}(i)<j$, and a departure otherwise. 
Thus, in total, we get a contribution $q^{a(D,\delta)}$ from crossings which are not switches.
\end{proof}

\subsection{Dualities}
\label{subsec_S_dualities}

The standard contact form $dz-ydx$ on $\RR^3$ is invariant under $(x,y,z)\mapsto(-x,-y,z)$ and switches sign under $(x,y,z)\mapsto(x,-y,-z)$, thus induce maps on legendrian curves denoted by $D_v$ and $D_h$ respectively.
We can upgrade these to maps of \textit{graded} legendrian curves by the following rule: For $D_v$ keep the same integers on the strands and for $D_h$ reverse the sign of the integers on the strands. 
This rule is essentially forced, up to overall shift, by the condition on the grading near cusps, c.f. the discussion on grading in the introduction.

We claim that the maps $D_v$ and $D_h$ are compatible with the skein relations \eqref{gs_1f}, \eqref{gs_2f}, and \eqref{gs_2f}.
This is fairly easy to see except perhaps for $D_v$ and \eqref{gs_1f}, but obvious for the relation equivalent to \eqref{gs_1f} (up to legendrian isotopy) found in the proof of Proposition~\ref{prop_phi}.
Thus $D_v$ and $D_h$ give functor from $\mc S$ to itself where $D_v$ acts as the identity on objects and is contravariant with respect to composition and covariant with respect to the monoidal product, and $D_h$ acts on objects by $(X,\deg)\mapsto(-X,-\deg)$ where $X\subset\RR$, $\deg:X\to\ZZ$, is covariant with respect to composition and contravariant with respect to monoidal product.
The functors are involutive and commute.
We summarize the discussion in the table below.

\begin{center}
\begin{tabular}{c|c|c}
 & $D_v$ & $D_h$ \\
\hline
mapping & $(x,y,z)\mapsto(-x,-y,z)$ & $(x,y,z)\mapsto(x,-y,-z)$ \\
front & flip on vertical axis & flip on horizontal axis \\
grading & $n\mapsto n$ & $n\mapsto -n$ \\
$\circ$ & contravariant & covariant \\
$\otimes$ & covariant & contravariant
\end{tabular}
\end{center}

\section{The functor $\Phi:\mc S|_q\to\mc H$}
\label{sec_functor}

To define a monoidal functor $\Phi$ from the category of tangles $\mc S$ it suffices to fix the value of $\Phi$ on objects, elementary tangles $\sigma_{m,n}$, $\lambda_n$, $\rho_n$ (see \eqref{elem_tangles}), and check invariance under Legendrian isotopy (planar isotopy and Reidemeister moves) and the skein relations.
Fix a prime power $q$ and let $\mc S|_q$ be the $\QQ$-linear category obtained by base change along the homomorphism $\ZZ[t^{\pm},(t-1)^{-1}]\to\QQ$ sending the formal variable $t$ to the given prime power $q$.
The same remarks about defining functors apply.

The functor $\Phi:\mc S_q\to\mc H$ is defined on objects as follows. 
An object of $\mc S$ is represented by a $\ZZ$-graded set of reals $X\subset\RR$, and $\Phi$ sends it to the graded vector space $V$ over $\FF_q$ with basis $X$ and complete flag such that $F_iV$ has basis the first (smallest) $i$ elements of $X$, $i=0,\ldots,|X|$.
To avoid writing lots of matrices, let $e_{ij}$ denote the matrix with all zeros except for a $1$ in the $i$-th row and $j$-th column.
$\Phi$ is defined on elementary tangles by
\begin{gather*}
\Phi(\sigma_{m,n}):=(q-1)^{-2}\left(0,e_{12}+e_{21},0\right),\\ 
\Phi(\lambda_n):=(q-1)^{-1}\left(e_{12},0,0\right),  \\
\Phi(\rho_n):=(q-1)^{-1}\left(0,0,e_{12}\right).
\end{gather*}
Since $\Phi$ should preserve identity morphisms we also have
\begin{equation*}
\Phi(1_n)=(q-1)^{-1}(0,1,0).
\end{equation*}

Our first task is to compute the value of $\Phi$ on basic tangles.
It will be convenient to introduce the symbol
\[
\tau_{m,n}:=\begin{cases} 1 & m>n \\ q^{(-1)^{m-n}} & m\leq n \end{cases}
\]
which depends only on the difference $m-n$.

\begin{lemma}\label{lem_basic_images}
Let $X,Y\in\Ob(\mc S)$ be $\ZZ$-graded subsets of $\RR$ and let $V=\Phi(X)$, $W=\Phi(Y)$ be the corresponding $\ZZ$-graded vector spaces with complete flag.
\begin{enumerate}
\item
The tangle $1_X=[0,1]\times X\times \{0\}$ maps to $1_V$, see \eqref{h_id}, under $\Phi$.
\item
The tangle $1_X\otimes\lambda_n\otimes 1_Y$ maps to
\begin{gather*}
(q-1)^{-\dim(V\oplus W)-1}\left(\prod_{i=0}^{\infty}\left|\Hom^{-i}_{<0}(V\oplus W,V\oplus W)\right|^{(-1)^{i+1}}\right)\cdot \\
\cdot\sum_B\left(\begin{bmatrix} B_{11} & 0 & B_{12} \\ 0 & \begin{bmatrix} 0 & 1 \\ 0 & 0 \end{bmatrix} & 0 \\ 0 & 0 & B_{22}\end{bmatrix},\begin{bmatrix} 1 & 0 & 0 \\ 0 & 0 & 1 \end{bmatrix},\begin{bmatrix} B_{11} & B_{12} \\ 0 & B_{22} \end{bmatrix} \right)
\end{gather*}
where the sum is over all matrices $B$ giving a filtration decreasing differential on $V\oplus W$.
\item 
Let $E:=\mathbf{k}[-n-1]\oplus\mathbf{k}[-n]$.
The tangle $1_X\otimes\rho_n\otimes 1_Y$ maps to
\begin{gather*}
(q-1)^{-\dim(V\oplus E\oplus W)}\left(\prod_{i=0}^{\infty}\left|\Hom^{-i}_{<0}(V\oplus E\oplus W,V\oplus E\oplus W)\right|^{(-1)^{i+1}}\right)\cdot \\
\cdot\sum_B\left(\begin{bmatrix} B_{11} & B_{13}-B_{12}B_{22}'B_{23} \\ 0 & B_{33} \end{bmatrix},\begin{bmatrix} 1 & 0 \\ 0 & -B_{22}'B_{23} \\ 0 & 1 \end{bmatrix},\begin{bmatrix} B_{11} & B_{12} & B_{13} \\ 0 & B_{22} & B_{23} \\ 0 & 0 & B_{33}\end{bmatrix} \right)
\end{gather*}
where the sum is over all matrices $B$ giving a filtration decreasing differential on $V\oplus E\oplus W$ and such that $B_{22}\neq 0$, thus $B_{22}=\begin{bmatrix} 0 & b \\ 0 & 0 \end{bmatrix}$ with $b\neq 0$ and we set $B_{22}':=\begin{bmatrix} 0 & 0 \\ 1/b & 0 \end{bmatrix}$.
\item
Let $m,n\in\ZZ$, $E:=\mathbf{k}[-m]\oplus\mathbf{k}[-n]$.
The tangle $1_X\otimes\sigma_{m,n}\otimes 1_Y$ maps to
\begin{gather*}
(q-1)^{-\dim(V\oplus E\oplus W)}\tau_{m,n}\left(\prod_{i=0}^{\infty}\left|\Hom^{-i}_{<0}(V\oplus E\oplus W,V\oplus E\oplus W)\right|^{(-1)^{i+1}}\right)\cdot \\
\cdot\sum_B\left(\begin{bmatrix} B_{11} & B_{12}T & B_{13} \\ 0 & 0 & TB_{23} \\ 0 & 0 & B_{33} \end{bmatrix},\begin{bmatrix} 1 & 0 & 0 \\ 0 & T & 0 \\ 0 & 0 & 1 \end{bmatrix},\begin{bmatrix} B_{11} & B_{12} & B_{13} \\ 0 & 0 & B_{23} \\ 0 & 0 & B_{33}\end{bmatrix} \right)
\end{gather*}
where $T=\begin{bmatrix} 0 & 1 \\ 1 & 0 \end{bmatrix}$ and the sum is over all matrices $B$ giving a filtration decreasing differential on $V\oplus E\oplus W$ and such that $B_{22}=0$.
\end{enumerate}
\end{lemma}

\begin{proof}
We begin by computing the monoidal product of a general morphism with a morphism of the form $(d,1,d)$.
We claim that
\begin{equation}
\begin{aligned}
\label{monoidal_prod_unit_left}
\left(d_U,1,d_U\right)&\otimes\left(d_X,g,d_Y\right)=\\
&=\left(\prod_{i=0}^{\infty}\left|\Hom^{-i}(Y,U)\right|^{(-1)^{i+1}}\right)
 \sum_{\delta_{22}}\left(\begin{bmatrix} d_U & \delta_{22}g \\ 0 & d_X \end{bmatrix},\begin{bmatrix} 1 & 0 \\ 0 & g \end{bmatrix},\begin{bmatrix} d_U & \delta_{22} \\ 0 & d_Y \end{bmatrix} \right)
\end{aligned}
\end{equation}
where the sum is over $\delta_{22}\in\Hom^1(Y,U)$ with $d_U\delta_{22}+\delta_{22}d_Y=0$, and 
\begin{equation}
\begin{aligned}
\label{monoidal_prod_unit_right}
\left(d_U,f,d_V\right)&\otimes\left(d_X,1,d_X\right)=\\
&=\left(\prod_{i=0}^{\infty}\left|\Hom^{-i}(X,U)\right|^{(-1)^{i+1}}\right)
 \sum_{\delta_{11}}\left(\begin{bmatrix} d_U & \delta_{11} \\ 0 & d_X \end{bmatrix},\begin{bmatrix} f & 0 \\ 0 & 1 \end{bmatrix},\begin{bmatrix} d_V & f\delta_{11} \\ 0 & d_X \end{bmatrix} \right)
\end{aligned}
\end{equation}
where the sum is over $\delta_{11}\in\Hom^1(X,U)$ with $d_U\delta_{11}+\delta_{11}d_X=0$.

To see \eqref{monoidal_prod_unit_left} apply the definition of $\otimes$ which gives
\begin{align*}
\left(d_U,1,d_U\right)&\otimes\left(d_X,g,d_Y\right)=\\
&=\left(\prod_{i=0}^{\infty}\left(\left|\Hom^{-i}(X,U)\right|\left|\Hom^{-i}(Y,U)\right|\left|\Hom^{-i-1}(X,U)\right|\right)^{(-1)^{i+1}}\right) \\
&\quad\cdot\sum_{\delta}\left(\begin{bmatrix} d_U & \delta_{11} \\ 0 & d_X \end{bmatrix},\begin{bmatrix} 1 & \delta_{12} \\ 0 & g \end{bmatrix},\begin{bmatrix} d_U & \delta_{22} \\ 0 & d_Y \end{bmatrix} \right).
\end{align*}
This becomes \eqref{monoidal_prod_unit_left} after noting that the $\Hom^{-i}(X,U)$ and $\Hom^{-i-1}(X,U)$ terms in the product cancel except for $\left|\Hom^0(X,U)\right|^{-1}$ and that a shear transformation on $X\oplus U$ can be applied to remove $\delta_{12}\in\Hom^0(X,U)$.
The proof of \eqref{monoidal_prod_unit_right} is similar.
We continue with the individual basic tangles.

1. This follows by induction from $1_V\otimes 1_W=1_{V\oplus W}$ (with $\dim W=1$), which in turn follows directly from \eqref{monoidal_prod_unit_left} or \eqref{monoidal_prod_unit_right}.

2. 
Using \eqref{monoidal_prod_unit_left} and \eqref{monoidal_prod_unit_right} we get
\begin{align*}
1_V\otimes\Phi(\lambda_n)&\otimes 1_W=(q-1)^{-\dim(V\oplus W)-1} \\
&\cdot\left(\prod_{i=0}^{\infty}\left(\left|\Hom^{-i}_{<0}(V,V)\right|\cdot \left|\Hom^{-i}_{<0}(W,W)\right|\cdot\left|\Hom^{-i}_{<0}(W,V\oplus E)\right|\right)^{(-1)^{i+1}}\right)\\
&\cdot\sum_{B,b_1,b_2}\left(\begin{bmatrix} B_{11} & 0 & B_{12} \\ 0 & \begin{bmatrix} 0 & 1 \\ 0 & 0 \end{bmatrix} & \begin{matrix}
b_1 \\ b_2 \end{matrix} \\ 0 & 0 & B_{22}\end{bmatrix},\begin{bmatrix} 1 & 0 & 0 \\ 0 & 0 & 1 \end{bmatrix},\begin{bmatrix} B_{11} & B_{12} \\ 0 & B_{22} \end{bmatrix}\right)
\end{align*}
where $E:=\mathbf{k}[-m]\oplus\mathbf{k}[-n]$.
Note that 
\[
\prod_{i=0}^{\infty}\left|\Hom^{-i}_{<0}(W,E)\right|^{(-1)^{i+1}}=\left|\Hom^0(W,k[-n])\right|^{-1}
\]
because of the telescoping product and that in the sum above we can choose 
\[
b_1\in\Hom^1(W,\mathbf k[-n-1])=\Hom^0(W,\mathbf k[-n])
\] 
arbitrary while $b_2=-b_1B_{22}$.
Applying a shear transformation on $E\oplus W$ we find that there are $|\Hom^0(W,\mathbf k[-n])|$ summands equivalent to ones with $b_1=b_2=0$, thus showing the claimed formula for $\Phi(1_X\otimes\lambda_n\otimes 1_Y)$.

3. With $B_{22}$ as in the statement of the lemma we can write
\[
\Phi(\rho_n)=(q-1)^{-\dim E}\sum_{B_{22}}\left(0,0,B_{22}\right).
\]
Again we apply \eqref{monoidal_prod_unit_left} and \eqref{monoidal_prod_unit_right} which after some rearranging of the products gives
\begin{align*}
1_V\otimes\Phi(\rho_n)&\otimes 1_W=(q-1)^{-\dim(V\oplus E\oplus W)}\left(\prod_{i=0}^{\infty}\left|\Hom^{-i}_{<0}(V\oplus E\oplus W,V\oplus E\oplus W)\right|^{(-1)^{i+1}}\right)\cdot \\
&\cdot\left|\Hom^0(W,\mathbf k[-n])\right|\sum_B\left(\begin{bmatrix} B_{11} & B_{13} \\ 0 & B_{33} \end{bmatrix},\begin{bmatrix} 1 & 0 \\ 0 & 0 \\ 0 & 1 \end{bmatrix},\begin{bmatrix} B_{11} & B_{12} & B_{13} \\ 0 & B_{22} & 0 \\ 0 & 0 & B_{33}\end{bmatrix} \right)
\end{align*}
but
\begin{align*}
\left|\Hom^0(W,\mathbf k[-n])\right|&\left(\begin{bmatrix} B_{11} & B_{13} \\ 0 & B_{33} \end{bmatrix},\begin{bmatrix} 1 & 0 \\ 0 & 0 \\ 0 & 1 \end{bmatrix},\begin{bmatrix} B_{11} & B_{12} & B_{13} \\ 0 & B_{22} & 0 \\ 0 & 0 & B_{33}\end{bmatrix} \right)\\
&=\sum_{B_{23}}\left(\begin{bmatrix} B_{11} & B_{13}''-B_{12}B_{22}'B_{23} \\ 0 & B_{33} \end{bmatrix},\begin{bmatrix} 1 & 0 \\ 0 & -B_{22}'B_{23} \\ 0 & 1 \end{bmatrix},\begin{bmatrix} B_{11} & B_{12} & B_{13}'' \\ 0 & B_{22} & B_{23} \\ 0 & 0 & B_{33}\end{bmatrix} \right)
\end{align*}
with $B_{13}'':=B_{13}+B_{12}B_{22}'B_{23}$, $B_{22}'$ as in the statement of the lemma, and $B_{23}=\begin{bmatrix} b_1 \\ b_2 \end{bmatrix}$ with $b_1\in\Hom^1(W,\mathbf k[-n-1])$ arbitrary and $b_2=-b^{-1}b_1B_{33}$.

4. This follows again from \eqref{monoidal_prod_unit_left} and \eqref{monoidal_prod_unit_right} and
\[
\prod_{i=0}^{\infty}\left|\Hom_{<0}^{-i}(E,E)\right|^{(-1)^i}=\prod_{i=0}^{\infty}\left|\Hom^{-i}(\mathbf{k}[-n],\mathbf{k}[-m])\right|^{(-1)^i}=\tau_{m,n}
\]
where $E=\mathbf{k}[-m]\oplus\mathbf{k}[-n]$ as in the statement of the lemma.
\end{proof}

From this, the value of $\Phi$ on a tangle which is presented as a horizontal composition of basic tangles is fixed by the requirement that $\Phi$ is a functor.
\begin{prop}\label{prop_phi}
The above rules define a unique functor $\Phi:\mc S|_q\to\mc H$ of monoidal categories.
This functor is compatible with the dualities: $\Phi\circ D_v\cong D\circ\Phi$ and $\Phi\circ D_h\circ D_v\cong D'\circ\Phi$ where $D'$ is the functor $V\mapsto V^\vee$ induced by vector space duality, see \eqref{Hdual_2}.
\end{prop}

\begin{proof}
Note that the value of $\Phi$ on elementary tangles is compatible with the dualities, so compatibility with dualities for general morphisms follows from their covariance/contravariance properties once we have shown that $\Phi$ is well-defined.

We first check invariance of $\Phi$ under planar isotopy.
Any planar isotopy is a composition of the following basic moves: a cusp/crossing passing over another cusp/crossing (i.e. switching the order of their projection to the $x$-axis).
Invariance of $\Phi$ thus follows immediately from the fact that $\mc H$ is monoidal, more specifically that identities of the form
\begin{equation*}
(a\otimes 1)\circ(1\otimes b)=a\otimes b=(1\otimes b)\circ(a\otimes 1)
\end{equation*}
or pictorially
\begin{equation*}
\begin{tabular}{|c|c|}
\hline 
$1$ & $b$ \\
\hline  
$a$ & $1$ \\
\hline
\end{tabular}
=
\begin{tabular}{|c|}
\hline 
$b$ \\
\hline  
$a$ \\
\hline
\end{tabular}
=
\begin{tabular}{|c|c|}
\hline 
$b$ & $1$ \\
\hline  
$1$ & $a$ \\
\hline
\end{tabular}
\end{equation*}
hold.
It remains to show that $\Phi$ is compatible with the Reidemeister moves R1, R2, R3, and the skein relations S1, S2, S3.
For R1 there are two variants obtained by reflection on the horizontal axis and for R2 there are four variants obtained by reflection on both the horizontal and vertical axes.
However, once we show one of these variants the others follow by duality in the category $\mc H$.

\textbf{(R1)} 
The first Reidemeister move is the identity
\begin{equation*}
(1_n\otimes\lambda_{n-1})(\sigma_{n,n}\otimes 1_{n-1})(1_n\otimes\rho_{n-1})=
\begin{tikzpicture}[baseline=-\dimexpr\fontdimen22\textfont2\relax,scale=.8]
\draw[thick] (0,-.5) to (1,-.5) to [out=0,in=180] (2,0) to [out=0,in=180] (2.7,.25) to [out=180,in=0] (2,.5) to (1,.5) to [out=180,in=0] (.3,.25) to [out=0,in=180] (1,0) to [out=0,in=180] (2,-.5) to (3,-.5);
\node[blue,above] at (.5,-.55) {\scriptsize{n}};
\node[blue,above] at (1.5,.45) {\scriptsize{n-1}};
\node[blue,above] at (2.5,-.55) {\scriptsize{n}};
\end{tikzpicture}%
=
\begin{tikzpicture}[baseline=-\dimexpr\fontdimen22\textfont2\relax,scale=.8]
\draw[thick] (0,0) to (1,0);
\node[blue,above] at (.5,0) {\scriptsize{n}};
\end{tikzpicture}%
=1_n
\end{equation*}
so we need to show that
\[
\Phi(1_n\otimes\lambda_{n-1})\Phi(\sigma_{n,n}\otimes 1_{n-1})\Phi(1_n\otimes\rho_{n-1})=\Phi(1_n)
\]
holds.
By Lemma~\ref{lem_basic_images} we have
\begin{align*}
\Phi(1_n)&=(q-1)^{-1}(0,1,0) \\
\Phi(1_n\otimes\lambda_{n-1})&=(q-1)^{-2}(e_{23},e_{11},0) \\
\Phi(1_n\otimes\rho_{n-1})&=(q-1)^{-2}(0,e_{11},e_{23}) \\
\Phi(\sigma_{n,n}\otimes 1_{n-1})&=(q-1)^{-3}\sum_{a,b}(be_{13}+ae_{23},T\oplus 1,ae_{13}+be_{23})\\
&= (q-1)^{-3}\left((0,T\oplus 1,0)+(q-1)(e_{23},T\oplus 1,e_{13})\right. \\
& \quad \left. + (q-1)(e_{13},T\oplus 1,e_{23})+(q-1)^2(e_{23},e_{11},e_{23})\right)
\end{align*}
where the four summands correspond to the cases $a=b=0$, $a\neq 0,b=0$, $a=0,b\neq 0$, and $a,b\neq 0$ respectively.
Note also that we use a homotopy as well as a change of basis to show the equivalence with $(e_{23},e_{11},e_{23})$.
Only the $(e_{23},e_{11},e_{23})$ term contributes to the product
\[
\Phi(1_n\otimes\lambda_{n-1})\Phi(\sigma_{n,n}\otimes 1_{n-1})\Phi(1_n\otimes\rho_{n-1})=(q-1)^{-5}(e_{23},e_{11},0)(e_{23},e_{11},e_{23})(0,e_{11},e_{23})
\]
which is equal to $(q-1)^{-1}(0,1,0)$ using the fact that the filtered complex $\mathbf k[-n]\oplus\mathbf k[-n]\oplus\mathbf k[-n+1]$ with differential $d=e_{23}$ has group of automorphisms $(\mathbf k^\times)^2$ whose size is $(q-1)^{2}$.

\textbf{(R2)} 
The second Reidemeister move is the identity
\[
(\sigma_{m,n+1}\otimes 1_n)(1_{n+1}\otimes\sigma_{m,n})(\rho_n\otimes 1_m)=
\begin{tikzpicture}[baseline=-\dimexpr\fontdimen22\textfont2\relax,scale=.8]
\draw[thick] (0,-.5) to [out=0,in=180] (1,0) to [out=0,in=180] (2,.5) to (3,.5);
\draw[thick] (0,0) to [out=0,in=180] (1,-.5) to (2,-.5) to [out=0,in=180] (2.7,-.25) to [out=180,in=0] (2,0) to [out=180,in=0] (1,.5) to (0,.5);
\node[blue,above] at (.5,.45) {\scriptsize{n}};
\node[blue,above] at (2.5,.45) {\scriptsize{m}};
\node[blue,below] at (1.5,-.45) {\scriptsize{n+1}};
\end{tikzpicture}%
=
\begin{tikzpicture}[baseline=-\dimexpr\fontdimen22\textfont2\relax,scale=.8]
\draw[thick] (0,-.5) to (1,-.5);
\draw[thick] (0,0) to [out=0,in=180] (.7,.25) to [out=180,in=0] (0,.5);
\node[blue,above] at (.5,.35) {\scriptsize{n}};
\node[blue,below] at (.5,-.45) {\scriptsize{m}};
\end{tikzpicture}%
=1_m\otimes\rho_n
\]
so we need to show that
\begin{equation}\label{r2_phi}
\Phi(\sigma_{m,n+1}\otimes 1_n)\Phi(1_{n+1}\otimes\sigma_{m,n})\Phi(\rho_n\otimes 1_m)=\Phi(1_m\otimes\rho_n)
\end{equation}
holds.
By Lemma~\ref{lem_basic_images} we have
\begin{align*}
\Phi(\rho_n\otimes 1_m)&=(q-1)^{-2}(0,e_{13},e_{12}) 
\end{align*}
\begin{align*}
\Phi(1_{n+1}\otimes \sigma_{m,n})&=(q-1)^{-3}\tau_{n+1,m}^{-1}\left(\sum_a(ae_{12},1\oplus T,ae_{13})\right. \\
&\quad\left.+\delta_{m,n}\sum_{a,b\neq 0}(ae_{12}+be_{13},1\oplus T,be_{12}+ae_{13})\right) \\
&=(q-1)^{-2}\tau_{n+1,m}^{-1}(e_{12},1\oplus T,e_{13})+\ldots 
\end{align*}
\begin{align*}
\Phi(\sigma_{m,n+1}\otimes 1_n)&=(q-1)^{-3}\tau_{m,n}^{-1}\left(\sum_a(ae_{13},T\oplus 1,ae_{23})\right. \\
&\quad\left.+\delta_{m,n+1}\sum_{a,b\neq 0}(ae_{13}+be_{23},T\oplus 1,be_{13}+ae_{23})\right) \\
&=(q-1)^{-2}\tau_{m,n}^{-1}(e_{13},T\oplus 1,e_{23})+\ldots
\end{align*}
where the ellipsis ($\ldots$) represents terms which do not contribute to the product~\eqref{r2_phi}.
Looking at the sizes of automorphism groups of filtered complexes in the various cases $m=n$, $m=n+1$, $m\neq n,n+1$ one finds that the product of the first two factors in \eqref{r2_phi} contributes a scalar factor $(q-1)^2\tau_{n+1,m}\tau_{m,n}$ and the product of the second and third factors in \eqref{r2_phi} contributes a scalar factor $(q-1)^2$.
It follows that the left-hand side of \eqref{r2_phi} is equal to
\[
\Phi(1_m\otimes\rho_n)=(q-1)^{-2}(0,e_{11},e_{23})
\]
as claimed.

\textbf{(R3)} 
The third Reidemeister move is the identity
\begin{gather*}
(1_k\otimes\sigma_{n,m})(\sigma_{k,m}\otimes 1_n)(1_m\otimes\sigma_{k,n})=
\begin{tikzpicture}[baseline=-\dimexpr\fontdimen22\textfont2\relax,scale=.8]
\draw[thick] (0,-.5) to (1,-.5) to [out=0,in=180] (2,0) to [out=0,in=180] (3,.5);
\draw[thick] (0,0) to [out=0,in=180] (1,.5) to (2,.5) to [out=0,in=180] (3,0);
\draw[thick] (0,.5) to [out=0,in=180] (1,0) to [out=0,in=180] (2,-.5) to (3,-.5);
\node[blue,below] at (.5,-.45) {\scriptsize{k}};
\node[blue,above] at (1.5,.45) {\scriptsize{n}};
\node[blue,below] at (2.5,-.45) {\scriptsize{m}};
\end{tikzpicture}%
\qquad\\
\qquad=
\begin{tikzpicture}[baseline=-\dimexpr\fontdimen22\textfont2\relax,scale=.8]
\draw[thick] (0,.5) to (1,.5) to [out=0,in=180] (2,0) to [out=0,in=180] (3,-.5);
\draw[thick] (0,0) to [out=0,in=180] (1,-.5) to (2,-.5) to [out=0,in=180] (3,0);
\draw[thick] (0,-.5) to [out=0,in=180] (1,0) to [out=0,in=180] (2,.5) to (3,.5);
\node[blue,above] at (.5,.45) {\scriptsize{m}};
\node[blue,below] at (1.5,-.45) {\scriptsize{n}};
\node[blue,above] at (2.5,.45) {\scriptsize{k}};
\end{tikzpicture}%
=(\sigma_{k,n}\otimes 1_m)(1_n\otimes\sigma_{k,m})(\sigma_{n,m}\otimes 1_k)
\end{gather*}
so we need to show that
\begin{equation}\label{r3_phi}
\Phi(1_k\otimes\sigma_{n,m})\Phi(\sigma_{k,m}\otimes 1_n)\Phi(1_m\otimes\sigma_{k,n})=\Phi(\sigma_{k,n}\otimes 1_m)\Phi(1_n\otimes\sigma_{k,m})\Phi(\sigma_{n,m}\otimes 1_k)
\end{equation}
holds.
We claim that both sides are equal to $(q-1)^{-3}(0,e_{13}+e_{22}+e_{31},0)$ and will show this for the left-hand side, the calculation for the right-hand side being similar.

To simplify the calculation, note first that while each factor on the left hand side of \eqref{r3_phi} potentially has terms with non-zero differential, these do not contribute to the final product.
Furthermore, a five-term product of matrices of the form
\[
(1\oplus T)B_1(T\oplus 1)B_2(1\oplus T)
\]
with $B_1,B_2$ being invertible upper-triangular 3-by-3 matrices, lies in the generic Bruhat-cell, i.e. becomes $e_{13}+e_{22}+e_{31}$ after multiplying by certain upper-triangular matrices on the left and right.
(This is because $(1 3)=(2 3)\circ (1 2) \circ (2 3)$ is a minimal factorization of the longest element in $S_3$ into simple transposition, but can be easily checked directly.)
It only remains to work out the scalar factor.
The horizontal product of the three factors contributes
\[
(q-1)^6\tau_{k,m}\tau_{k,n}\tau_{m,n}\tau_{m,k}\tau_{m,n}\tau_{k,n}
\]
while from the vertical product used to form each factor (or Lemma~\ref{lem_basic_images}) we get
\[
(\tau_{k,m}\tau_{k,n}\tau_{m,n}\tau_{m,k}\tau_{m,n}\tau_{k,n})^{-1}
\]
which shows that the end result comes with a factor $(q-1)^{-3}$.

\textbf{(S1)} 
Modulo Reidemeister moves, the first skein relation can also be written as
\begin{equation*}
\begin{tikzpicture}[baseline=-\dimexpr\fontdimen22\textfont2\relax,scale=.6]
\draw[thick] (-1,-1) to [out=60,in=-120] (-.6,-.15);
\draw[thick] (-.35,.15) to [out=45,in=180] (0,.25) to [out=0,in=120] (1,-1);
\draw[thick] (-1,1) to [out=-60,in=180] (0,-.25) to [out=0,in=-135] (.35,-.15);
\draw[thick] (.6,.15) to [out=60,in=-120] (1,1);
\node[blue,below] at (0,-.2) {\scriptsize{n}};
\node[blue,above] at (0,.2) {\scriptsize{m}};
\draw[dashed] (0,0) circle [radius=1.414];
\end{tikzpicture}%
-q^{(-1)^{m-n}}\;
\begin{tikzpicture}[baseline=-\dimexpr\fontdimen22\textfont2\relax,scale=.6]
\draw[thick] (-1,-1) to [out=60,in=180] (0,-.25) to [out=0,in=120] (1,-1);
\draw[thick] (-1,1) to [out=-60,in=180] (0,.25) to [out=0,in=-120] (1,1);
\node[blue,below] at (0,-.2) {\scriptsize{m}};
\node[blue,above] at (0,.2) {\scriptsize{n}};
\draw[dashed] (0,0) circle [radius=1.414];
\end{tikzpicture}%
=\delta_{m,n}(q-1)\;
\begin{tikzpicture}[baseline=-\dimexpr\fontdimen22\textfont2\relax,scale=.6]
\draw[thick] (-1,-1) to [out=60,in=-150] (-.2,-.1);
\draw[thick] (.2,.1) to [out=30,in=-120] (1,1);
\draw[thick] (-1,1) to [out=-60,in=150] (0,0) to [out=-30,in=120] (1,-1);
\node[blue,rotate=-25,below] at (-.5,-.6) {\scriptsize{n}};
\node[blue,rotate=25,above] at (-.5,.6) {\scriptsize{n}};
\draw[dashed] (0,0) circle [radius=1.414];
\end{tikzpicture}%
-\delta_{m,n+1}(1-q^{-1})\;
\begin{tikzpicture}[baseline=-\dimexpr\fontdimen22\textfont2\relax,scale=.6]
\draw[thick] (-1,-1) to [out=60,in=180] (-.3,0) to [out=180,in=-60] (-1,1);
\draw[thick] (1,-1) to [out=-120,in=0] (.3,0) to [out=0,in=-120] (1,1);
\node[blue,rotate=-25,above] at (.5,.6) {\scriptsize{n}};
\node[blue,rotate=25,above] at (-.5,.6) {\scriptsize{n}};
\draw[dashed] (0,0) circle [radius=1.414];
\end{tikzpicture}%
\end{equation*}
We verify the cases $m=n$, $m=n+1$, and $m\neq n,n+1$ separately.
For $n=m$ we compute
\begin{align*}
\Phi(\sigma_{n,n})\Phi(\sigma_{n,n})&=(q-1)^{-4}(0,T,0)(0,T,0) \\
&=(q-1)^{-2}(0,1,0)+(q-1)^{-1}(0,T,0) \\
&=q\Phi(1_n\otimes 1_n)+(q-1)\Phi(\sigma_{n,n}),
\end{align*}
for $m=n+1$ we have
\begin{align*}
\Phi(1_{n+1}\otimes 1_n)&=(q-1)^{-2}(0,1,0)+(q-1)^{-1}(e_{12},1,e_{12}) \\
&=q\Phi(\sigma_{n+1,n})\Phi(\sigma_{n,n+1})+(q-1)\Phi(\rho_n)\Phi(\lambda_n),
\end{align*}
while for $m\neq n,n+1$ we have
\begin{align*}
\Phi(\sigma_{m,n})\Phi(\sigma_{n,m})=(q-1)^{-2}\tau_{n,m}(0,1,0)=q^{(-1)^{m-n}}\Phi(1_m\otimes 1_n)
\end{align*}
where we used that $\tau_{m,n}\tau_{n,m}=q^{(-1)^{m-n}}$ for $m\neq n$.

\textbf{(S2)} 
The claim is that 
\[
\Phi(\lambda_n\otimes 1_{n-1})\Phi(1_{n+1}\otimes\rho_{n-1})=0
\]
which follows from the fact that the filtered complexes $(C,d)$ and $(C,d')$ with
\[
C:=\mathbf k[-n-1]\oplus\mathbf k[-n]\oplus\mathbf k[-n+1],\qquad d:=e_{12},\qquad d':=e_{23}
\]
are not isomorphic.

\textbf{(S3)} 
To see that
\[
\Phi(\lambda_n)\Phi(\rho_n)=(q-1)^{-2}(e_{12},0,0)(0,0,e_{12})=(q-1)^{-1}
\]
note that the elementary two-step complex has group of automorphisms $\mathbf k^\times$.
\end{proof}

\begin{prop}\label{prop_phi_rulings}
Let $X\in\mathrm{Ob}(\mc S)$. 
Under the identification $\Hom_{\mc H}(\Phi(X),0)\cong \QQ^{\mc R(X)}$ (see Subsection~\ref{subsec_flags}) we have
\begin{equation}
\Phi=\nu:\Hom_{\mc S|_q}(X,\emptyset)\to\Hom_{\mc H}(\Phi(X),0)
\end{equation}
i.e. the value of $\Phi$ on a one-sided tangle is equal to a weighted count of rulings.
\end{prop}

\begin{proof}
Suppose $L$ is a tangle with $\partial_0L=\emptyset$ and $\partial_1L=X$.
We prove by induction on the number of crossings and cusps of $L$ that $\Phi(L)=\nu(L)$.
Let $B$ be a basic tangle (i.e. with a single cusp or crossing) such that $\partial_0B=X$ and set $X':=\partial_1B$.
We get a diagram
\begin{equation}\label{add_btangle}
\begin{tikzcd}
\Hom_{\mc S|_q}(X,\emptyset) \arrow{r}{\circ B}\arrow{d}{\nu} & \Hom_{\mc S|_q}(X',\emptyset) \arrow{d}{\nu} \\
\Hom_{\mc H}(\Phi X,0) \arrow{r}{\circ\Phi B} & \Hom_{\mc H}(\Phi X',0)
\end{tikzcd}
\end{equation}
which we need to show commutes since this would imply
\[
\Phi(L\circ B)=\Phi L\circ \Phi B=\nu(L)\circ \Phi B=\nu(L\circ B)
\]
proving the claim for the tangle $L\circ B$ with one more cusp/crossing than $L$.
The lower arrow in \eqref{add_btangle} can be described more explicitly with the help of Lemma~\ref{lem_basic_images} which gives formulas for $\Phi B$.
On the other hand, the map which follows the diagram from the lower-left corner to the lower-right corner via the upper path is described as follows: Take a ruling $(D,\delta)$ of $X$, then it maps to $\sum_{\rho} c_\rho\partial_1\rho$ where $\rho$ ranges over all rulings of $B$ which extend $(D,\delta)$ to $B$.
Here we are using a slight generalization of the notion of a ruling for tangles with possibly both boundaries non-empty. 
We will describe this map explicitly in each of the three cases, depending on the type of $B$.

Suppose first that $B$ has a left cusp, so $B=1_Y\otimes \lambda_n\otimes 1_Z$ where $Y\otimes Z=X$.
The map
\[
\QQ^{\mc R(Y\otimes Z)}\to\QQ^{\mc R(Y\otimes\partial_1\lambda_n\otimes Z)}
\]
is given by
\[
(D,\delta)\mapsto (q-1)^{-1}(D',\delta')
\]
where, if $\partial_0\rho_n=\{y,z\}$ with $y<z$, then $(D',\delta')$ is the extension of the ruling $(D,\delta)$ of $Y\otimes Z$ with $D'=D\cup\{z\}$ and $\delta'(z)=y$.
It follows from Lemma~\ref{lem_basic_images} that composition with $\Phi B$ on the right has the same effect, i.e. maps $(d(D,\delta),0,0)$ to $(q-1)^{-1}(d(D',\delta'),0,0)$.
To see this, note that in the formula for $\Phi B$ the set of triples which appear is preserved under the action of filtration-preserving automorphisms on the target $V\oplus W$. 
Thus in the formula for the composition in $\mc H$ we do not need to take the sum over all $b$ but just one particular one and add an overall factor counting the number of filtration preserving isomorphisms.
Finally note that if we set $B=d(D,\delta)$ in the triple in the formula of $\Phi B$ from Lemma~\ref{lem_basic_images}, then the differential on the source is $d(D',\delta')$.

Suppose that $B$ has a right cusp, so $B=1_Y\otimes \rho_n\otimes 1_Z$ where $Y\otimes \partial_0\rho_n\otimes Z=X$.
Let $\partial_0\rho_n=\{y,z\}$ with $y<z$.
The map
\[
\QQ^{\mc R(Y\otimes\partial_0\rho_n\otimes Z)}\to\QQ^{\mc R(Y\otimes Z)}
\]
sends rulings with $z\in D$ and $\delta(z)=y$ to their restriction to $Y\otimes Z$ and all other rulings to zero.
From Lemma~\ref{lem_basic_images} we see by the same argument as in the previous case that composition with $\Phi B$ on the right has the same effect.

Finally, suppose that $B$ has a crossing, so $B=1_Y\otimes \sigma_{m,n}\otimes 1_Z$ where $X=Y\otimes \partial_0\sigma_{m,n}\otimes Z$.
Let $\partial_0\sigma_{m,n}=\{y,z\}$ with ordering $y<z$ and  $\partial_1\sigma_{m,n}=\{z,y\}$ with ordering $z<y$. 
The map
\[
\QQ^{\mc R(Y\otimes\partial_0\sigma_{m,n}\otimes Z)}\to\QQ^{\mc R(Y\otimes\partial_1\sigma_{m,n}\otimes Z)}
\]
is described as follows.
In the case $m\neq n$ a ruling $(D,\delta)$ gets sent to zero if $z\in D$ and $\delta(z)=y$ and to $\tau_{m,n}(D,\delta)$ otherwise.
In the case $m=n$ a ruling $(D,\delta)$ gets sent to $q(D,\delta)$ if when extending the ruling to $B$ without a switch, the crossing becomes a return and to $(D,\delta)+(q-1)(D',\delta')$ if the crossing becomes a departure, where $(D',\delta')$ is obtained from $(D,\delta)$ by switching the roles of $y$ and $z$, i.e. is the same ruling if the two sets are identified via the order preserving bijection.
From Lemma~\ref{lem_basic_images} we see that composition with $\Phi B$ on the right has the same effect.
Unlike the previous cases, the set triples which appear in the formula for $\Phi B$ is not invariant under flag preserving automorphism of the target, but this becomes true when instead of using the fixed $T$ we take the sum over all elements of the form $Tb$ where $b$ is a two-by-two upper triangular matrix.
The same reasoning as in the previous cases then works.
The computation is essentially the same as the one which verified that the skein relation S1 holds in $\mc H$ in the proof of Proposition~\ref{prop_phi}.
\end{proof}

We previously introduced the tangle $\beta_Y\in\Hom_{\mc S}(Y,\emptyset)$ and the morphism $\beta_W\in\Hom_{\mc H}(W,0)$.
As the notation suggests, these correspond to one another under $\Phi$.

\begin{prop}
Let $Y\in\Ob(\mc S)$, then $\Phi(\beta_Y)=\beta_{\Phi Y}$.
\end{prop}

\begin{proof}
This follows from Lemma~\ref{lem_beta_rulings}, where $\nu(\beta_Y)$ was computed, and Proposition~\ref{prop_phi_rulings}.
Note that it follows from the discussion in Subsection~\ref{subsec_flags} that
\[
\sum_d\left(\begin{bmatrix} -d & 1 \\ 0 & d\end{bmatrix},0,0\right)=\sum_{(D,\delta)\in\mc{R}_{\mathrm{part}}(Y)}(q-1)^{|D|}q^{s(D,\delta)}\left(\begin{bmatrix} -d(D,\delta) & 1 \\ 0 & d(D,\delta) \end{bmatrix},0,0\right)
\]
where the sum on the left ranges over differentials on $\Phi Y$.
\end{proof}

\begin{coro}
\label{cor_2to1commutes}
Let $X,Y\in\Ob(\mc S)$ then the diagram
\begin{equation}\label{one_two_sided_diagram}
\begin{tikzcd}
\Hom_{\mc S|_q}(X,Y) \arrow{r}{\Phi}\arrow{d} & \Hom_{\mc H}(\Phi X,\Phi Y) \arrow{d} \\
\Hom_{\mc S|_q}(Y[-1]\otimes X,\emptyset) \arrow{r}{\Phi} & \Hom_{\mc H}(\Phi Y[-1]\otimes \Phi X,0)
\end{tikzcd}
\end{equation}
where the left vertical arrow is the isomorphism~\eqref{tangle_bend_iso}, mapping $L$ to $\beta_Y\circ(1_{Y[-1]}\otimes L)$, and the right vertical arrow is the isomorphism~\eqref{catH_cone} mapping $(d_V,f,d_W)$ to $(\mathrm{Cone}(f)[-1],0,0)$, commutes.
\end{coro}

\begin{proof}
Follows from the previous proposition and Lemma~\ref{lem_Hcone} which says that the right vertical arrow is also given by $f\mapsto \beta_{\Phi Y}\circ(1_{\Phi Y[-1]}\otimes f)$.
\end{proof}

\begin{proof}[Proof of Theorem~\ref{thm_phi_equivalence}]
$\Phi$ is essentially surjective, since to any graded vector space with flag $V\in\mathrm{Ob}(\mc H)$ we can assign $\Phi^{-1}V\in\mathrm{Ob}(\mc S)$ where $\Phi^{-1}V:=\{1,\ldots,\dim V\}$ with grading $\deg(i)=\deg F_iV/F_{i-1}V$ as before.

For $X,Y\in\mathrm{Ob}(\mc S)$ we need to show that $\Phi$ gives an isomorphism from $\Hom_{\mc S|_q}(X,Y)$ to $\Hom_{\mc H}(\Phi X,\Phi Y)$.
By Corollary~\ref{cor_2to1commutes} it suffices to consider the special case $Y=\emptyset$, but then the claim follows from Proposition~\ref{prop_nu_iso} and Proposition~\ref{prop_phi_rulings}.
\end{proof}

\appendix

\section{Quiver representations}
\label{sec_qrep}

In this section we provide an explicit dg-model for the derived category of representations of a quiver $Q$.
This particular model has the virtue of having small $\Hom$-complexes and comes from computing $\Ext^\bullet(E,F)$ by replacing $E$ by its minimal projective resolution.
It is well-known to experts, though we could not find a suitable reference.
The dg-category will be used as a starting point for the categories defined in the following section.

Fix a (finite) quiver $Q$ with set of vertices $Q_0$ and an arbitrary field $\mathbf k$.
Define a dg-category $\mc D(Q)$ of complexes of quiver representations.
An object of $\mc D(Q)$ is given by a chain complex $(E_i,d)$ (over $\mathbf k$) for each vertex $i\in Q_0$ and a chain map $T_\alpha:E_i\to E_j$ of degree 0 for each arrow $\alpha:i\to j$.
Given a pair of objects $E=\left((E_i)_i,(S_\alpha)_\alpha\right)$ and $F=\left((F_i)_i,(T_\alpha)_\alpha\right)$ let
\begin{equation*}
\Hom_{\mc D(Q)}(E,F):=\bigoplus_{i\in Q_0}\Hom(E_i,F_i)\oplus\bigoplus_{i\xrightarrow{\alpha}j}\Hom(E_i[1],F_j)
\end{equation*}
where $\Hom(E_i,F_i)=\bigoplus_k\Hom^k(E_i,F_i)$ includes homogeneous maps of all degrees.
The differential is given by
\begin{align*}
d\left((f_i)_i,(g_\alpha)_\alpha\right):&=\left( \left(d\circ f_i-(-1)^{|f_i|}f_i\circ d\right)_i, \right. \\
 &\quad\quad\left.\left((-1)^{|f_i|}T_\alpha\circ f_i-(-1)^{|f_j|}f_j\circ S_\alpha+d\circ g_\alpha+(-1)^{|g_\alpha|}g_\alpha\circ d\right)_\alpha\right) \nonumber
\end{align*}
where $\alpha:i\to j$.
Composition is defined by
\begin{equation*}
\left((f_i)_i,(g_\alpha)_\alpha\right)\circ\left((f_i')_i,(g_\alpha')_\alpha\right)=\left((f_i\circ f_i')_i,(f_j\circ g_\alpha+(-1)^{|f_i|}g_\alpha\circ f_i)\right).
\end{equation*}

The proof of the following proposition is a straightforward checking of signs and will be omitted.

\begin{prop}
$\mc D(Q)$ is a dg-category.
\end{prop}

\section{Counting in higher categories}
\label{sec_hcard}

A good definition of the cardinality of an essentially finite groupoid $\mathcal G$ is
\begin{equation*}
\sum_{X\in\mathrm{Iso}(\mc G)}\frac{1}{|\Aut(X)|}
\end{equation*}
where $\mathrm{Iso}(\mc G)=\Ob(\mc G)/\cong$ is the set of isomorphism classes of objects in $\mc G$ and $\Aut(X)=\Hom_{\mc G}(X,X)$.
This idea was generalized to $\infty$-groupoids, aka homotopy types, by Baez--Dolan~\cite{baez_dolan} and we will apply it to dg-categories following Toen~\cite{toen_derived_hall}.

Let $\mc C$ be a dg-category over a finite field $\FF_q$ such that for each $X,Y\in\mathrm{Ob}(\mathcal C)$ all $\mathrm{Ext}^i(X,Y):=H^i(\Hom^\bullet_{\mc C}(X,Y))$ are finite-dimensional and vanish for $i\ll 0$.
In this situation one can define the \textbf{counting measure}, $\mu_\#$, on the set of isomorphism classes $\mathrm{Iso}(\mathcal C)$ by
\begin{equation*}
\mu_\#(\{X\}):=\left|\mathrm{Aut}(X)\right|^{-1}\prod_{i=1}^\infty\left|\mathrm{Ext}^{-i}(X,X)\right|^{(-1)^{i+1}}.
\end{equation*}
The set $\QQ\mathrm{Iso}(\mathcal C)$ of finite $\QQ$-linear combinations of elements in $\mathrm{Iso}(\mc C)$ has two interpretations: one as finitely supported functions on $\mathrm{Iso}(\mathcal C)$ and another as finite signed measures on $\mathrm{Iso}(\mathcal C)$.
We will adopt the latter here, following Kontsevich--Soibelman~\cite[Section 6.1]{ks}.
Note that Toen~\cite{toen_derived_hall} uses the former convention.
The conversion factor between the two is given by $\mu_\#$.

Suppose $F:\mathcal C\to\mathcal D$ is a dg-functor where $\mathcal C$ and $\mathcal D$ are $\FF_q$-linear and satisfy the finiteness conditions as in the previous paragraph.
Assume furthermore that $F$ sends only finitely many distinct $X\in\mathrm{Iso}(\mathcal C)$ to a given $Y\in\mathrm{Iso}(\mathcal D)$.
Define induced linear maps $F_*:\QQ\mathrm{Iso}(\mc C)\to \QQ\mathrm{Iso}(\mc D)$ and $F^!:\QQ\mathrm{Iso}(\mc D)\to \QQ\mathrm{Iso}(\mc C)$ by
\begin{gather*}
(F_*f)(Y):=\sum_{\substack{X\in\mathrm{Iso}(\mc C) \\ F(X)=Y}}f(X) \\
(F^!f)(X):=f(F(X))\frac{\left|\mathrm{Aut}(FX)\right|}{\left|\mathrm{Aut}(X)\right|}\prod_{i=1}^\infty\left(\frac{\left|\Ext^{-i}(FX,FX)\right|}{\left|\Ext^{-i}(X,X)\right|}\right)^{(-1)^i}.
\end{gather*}
The maps $F_*$, $F^!$ fit with the ``measures'' interpretation of $\QQ\mathrm{Iso}(\mc C)$.
For the ``functions'' interpretation it is more natural to use $F^*:\QQ\mathrm{Iso}(\mathcal D)\to \QQ\mathrm{Iso}(\mathcal C)$ and $F_!:\QQ\mathrm{Iso}(\mathcal C)\to \QQ\mathrm{Iso}(\mathcal D)$ defined by
\begin{gather*}
(F^*f)(X):=f(F(X)) \\
(F_!f)(Y):=\sum_{\substack{X\in\mathrm{Iso}(\mc C) \\ F(X)=Y}}f(X)\frac{\left|\Aut(Y)\right|}{\left|\Aut(X)\right|}\prod_{i=1}^\infty\left(\frac{\left|\Ext^{-i}(Y,Y)\right|}{\left|\Ext^{-i}(X,X)\right|}\right)^{(-1)^i}
\end{gather*}
so that
\begin{equation*}
F^!(f\mu_\#)=(F^*f)\mu_\#,\qquad F_*(f\mu_\#)=(F_!f)\mu_\#.
\end{equation*}

There is a simpler formula for $F^!$ for a special class of functors which was proven in \cite{haiden_skeinhall}.
More precisely, in addition to the finiteness conditions already imposed, we will assume that:
\begin{enumerate}[1)]
\item 
$F$ is full at the chain level, i.e. the maps $\Hom^\bullet_{\mc C}(X,Y)\to\Hom^\bullet_{\mc D}(FX,FY)$ are surjective.
\item
$F$ has the isomorphism lifting property: Given an isomorphism $f:FX\to Y$ in $\mc D$ there is an object $\widetilde{Y}\in\Ob(\mc C)$ with $F\widetilde Y=Y$ and an isomorphism $\tilde f:X\to \widetilde Y$ with $F(\tilde f)=f$.
\item
$F$ reflects isomorphisms: If $F(f):FX\to FY$ is an isomorphism then $f$ is an isomorphism.
\end{enumerate}
(Here an \textit{isomorphism} is a map which is invertible up to homotopy.)
By the first assumption on $F$ we have an exact sequence of cochain complexes
\begin{equation*}
0\longrightarrow K^\bullet(X,Y)\longrightarrow \Hom^\bullet_{\mc C}(X,Y)\longrightarrow\Hom^\bullet_{\mc D}(FX,FY)\longrightarrow 0
\end{equation*}
for each $X,Y\in\Ob(\mc C)$, where 
\begin{equation*}
K^i(X,Y):=\mathrm{Ker}\left(\Hom^i_{\mc C}(X,Y)\to\Hom^i_{\mc D}(FX,FY)\right)
\end{equation*}
and thus long exact sequences
\begin{equation}\label{full_functor_les}
\ldots\longrightarrow HK^i(X,Y)\longrightarrow \Ext^i_{\mc C}(X,Y)\longrightarrow\Ext^i_{\mc D}(FX,FY)\longrightarrow\ldots.
\end{equation}
Given $Y\in\mc D$ let $F_Y$ be the set of equivalence classes of objects $X\in\mc C$ with $FX=Y$ where $X\sim X'$ if there is an isomorphism $f:X\to X'$ with $F(f)=1_Y$ in $\Hom^0(Y,Y)$ (equivalently: $F(f)=1_Y$ in $\Ext^0(Y,Y)$).

\begin{lemma}
\label{lem_upper_shriek_formula}
Let $F:\mc C\to\mc D$ be an dg-functor satisfying the above conditions, then
\begin{equation*}
F^!(Y)=\sum_{X\in F_Y}\prod_{i=0}^\infty\left|HK^{-i}(X,X)\right|^{(-1)^{i+1}}X
\end{equation*}
for any $Y\in\QQ\mathrm{Iso}(\mc D)$.
\end{lemma}

See~\cite{haiden_skeinhall} for the proof in the more general case of $A_\infty$-categories.

\subsection{Hall algebra}
\label{subsec_hall}

The monoidal product in the category $\mc H$ turns out to be a special case of the Hall algebra product for dg-categories, which in particular implies its associativity by general results which we recall in this subsection.

Let $\mc C$ be a dg-category over a finite field $\FF_q$ satisfying the finiteness conditions as above. 
Assume furthermore that $\mc C$ is triangulated, i.e. closed under shifts, cones, and has a zero object.
Then we have a diagram of categories and functors
\begin{equation*}
\begin{tikzcd}
 & \mc C^{\mc A_2} \arrow[dl,"F"']\arrow[dr,"G"] & \\
\mc C\times \mc C & & \mc C
\end{tikzcd}
\end{equation*}
where $\mc C^{\mc A_2}$ is the category of exact triangles in $\mc C$, whose objects can be concretely represented by twisted complexes $C\xrightarrow{\delta} A$, $\delta\in\Hom^1(C,A)$, $m_1(\delta)=0$, which $F$ sends to the pair $(A,C)$ and $G$ sends to $\mathrm{Cone}(C[-1]\xrightarrow{\delta} A)$.
Passing to $\QQ\mathrm{Iso}(\mc C)$, the pull--push along the diagram gives a product map
\begin{equation*}
\QQ\mathrm{Iso}(\mc C)\otimes \QQ\mathrm{Iso}(\mc C)\xrightarrow{\quad F^!\quad }\QQ\mathrm{Iso}(\mc C^{\mc A_2})\xrightarrow{\quad G_*\quad}\QQ\mathrm{Iso}(\mc C).
\end{equation*} 

Lemma~\ref{lem_upper_shriek_formula} provides the following explicit formula for the product.
\begin{equation}\label{hall_product}
A\cdot C=\left(\prod_{i=0}^\infty\left|\Ext^{-i}(C,A)\right|^{(-1)^{i+1}}\right)\sum_{f\in\Ext^1(C,A)}\mathrm{Cone}(C[-1]\xrightarrow{f} A)
\end{equation}
The vector space $\QQ\mathrm{Iso}(\mc C)$ together with this product is called the \textbf{Hall algebra} of $\mc C$, denoted $\mathrm{Hall}(\mc C)$.
This is an associative algebra, see \cite{toen_derived_hall} or \cite{haiden_skeinhall} for a short proof, with unit the class of the zero object $0\in\Ob(\mc C)$.

\subsection{Deriving the formula for composition in $\mc H$}
\label{subsec_comp_conceptual}

Fix an arbitrary field $\mathbf k$.
We begin by defining dg-categories $\mc F_n$ for each integer $n\geq 1$, whose objects are roughly speaking given by an object of $\mathrm{Perf}(\mathbf k)$ with $n$ complete flags, in the sense of triangulated categories, on it.
Formally, an object of $\mc F_n$ is given by $n$ chain complexes with complete flags $(C_i,d,F_\bullet C_i)$, $i=1,\ldots,n$ and quasi-isomorphisms $\phi_i:C_i\to C_{i+1}$, $i=1,\ldots,n-1$.
Forgetting the flags, the $C_i$ and $\phi_i$ given an object in $\mc D(A_n)$, where $A_n$ is the quiver with underlying graph the $A_n$ Dynkin diagram and all arrows oriented in the same direction.
A morphism of degree $k$ in $\mc D(A_n)$ has components $f_i:C_i\to D_i$, $i=1,\ldots,n$, of degree $k$ and $g_i:C_i\to D_i$, $i=1,\ldots,n-1$ of degree $k-1$ fitting into a diagram
\[
\begin{tikzcd}
C_1 \arrow{r}{\phi_1}\arrow{d}{f_1}\arrow{dr}{g_1} & C_2 \arrow{r}{\phi_2}\arrow{d}{f_2}\arrow{dr}{g_2} & \cdots\arrow{r}{\phi_{n-1}}\arrow{dr}{g_{n-1}} & C_n \arrow{d}{f_n} \\
D_1 \arrow{r}{\psi_1} & D_2 \arrow{r}{\psi_2} & \cdots\arrow{r}{\psi_{n-1}} & D_n
\end{tikzcd}
\]
were for a closed morphism each square commutes up to chain homotopy given by $g_i$.
We define morphisms in $\mc F_n$ to be the subset of those morphisms in $\mc D(A_n)$ such that $f_i(F_jC_i)\subseteq F_jD_i$.
The differential and composition in $\mc F_n$ are defined as for $\mc D(A_n)$.

Assume from now on that $\mathbf k=\FF_q$ is a finite field.
To make contact with the definition of $\mc H$ consider for given $V_1,\ldots,V_n\in\Ob(\mc H)$ the full subcategory $\mc F_n(V_1,\ldots,V_n)\subset \mc F_n$ of objects with $(C_i,F_\bullet C_i)=V_i$, then 
\[
\mathrm{Iso}(\mc F_2(V,W))=\mc B(V,W),\qquad \QQ\mathrm{Iso}(\mc F_2(V,W))=\Hom_{\mc H}(V,W)
\]
essentially by definition.
For a triple $U,V,W\in\Ob(\mc H)$ there is a diagram of functors
\[
\begin{tikzcd}
 & \mc F_3(U,V,W)\arrow{dl}[swap]{P}\arrow{dr}{Q} &  \\
\mc F_2(V,W)\times \mc F_2(U,V) & & \mc F_2(U,W)           
\end{tikzcd}
\]
where
\begin{gather*}
P\left(C_1\xrightarrow{\phi_1} C_2\xrightarrow{\phi_2}C_3\right):=\left(C_2\xrightarrow{\phi_2}C_3,C_1\xrightarrow{\phi_1} C_2\right) \\
 Q\left(C_1\xrightarrow{\phi_1} C_2\xrightarrow{\phi_2}C_3\right):=C_1\xrightarrow{\phi_2\circ\phi_1}C_3
\end{gather*}
and the pull-push along the diagram gives a map
\[
\QQ\mathrm{Iso}(\mc F_2(V,W))\otimes \QQ\mathrm{Iso}(\mc F_2(U,V))\xrightarrow{P^!} \QQ\mathrm{Iso}(\mc F_3(U,V,W))\xrightarrow{Q_*}\QQ\mathrm{Iso}(\mc F_2(U,W))
\]
which, by Lemma~\ref{lem_upper_shriek_formula} has the explicit formula \eqref{h_comp} which is the composition in $\mc H$.

\bibliographystyle{alpha}
\bibliography{skeinhall}

\begin{thebibliography}{{Rut}06}

\bibitem[{Bar}94]{barannikov94}
S.~A. {Barannikov}.
\newblock {The framed Morse complex and its invariants.}
\newblock In {\em {Singularities and bifurcations. Transl. ed. by A. B.
  Sossinsky}}, pages 93--115. Providence, RI: American Mathematical Society,
  1994.

\bibitem[BD01]{baez_dolan}
John~C. {Baez} and James {Dolan}.
\newblock {From finite sets to Feynman diagrams.}
\newblock In {\em {Mathematics unlimited---2001 and beyond}}, pages 29--50.
  Berlin: Springer, 2001.

\bibitem[{Eli}87]{eliashberg87}
Ya.~M. {Eliashberg}.
\newblock {A theorem on the structure of wave fronts and its applications in
  symplectic topology.}
\newblock {\em {Funct. Anal. Appl.}}, 21(1-3):227--232, 1987.

\bibitem[FY89]{freyd_yetter89}
Peter~J. {Freyd} and David~N. {Yetter}.
\newblock {Braided compact closed categories with applications to low
  dimensional topology.}
\newblock {\em {Adv. Math.}}, 77(2):156--182, 1989.

\bibitem[Hai]{haiden_skeinhall}
Fabian Haiden.
\newblock Legendrian skein algebras and hall algebras.
\newblock arXiv:1908.10358.

\bibitem[{Jon}87]{jones87}
V.~F.~R. {Jones}.
\newblock {Hecke algebra representations of braid groups and link polynomials.}
\newblock {\em {Ann. Math. (2)}}, 126:335--388, 1987.

\bibitem[JS93]{joyal_street93}
Andr\'e {Joyal} and Ross {Street}.
\newblock {Braided tensor categories.}
\newblock {\em {Adv. Math.}}, 102(1):20--78, 1993.

\bibitem[KS]{ks}
Maxim Kontsevich and Yan Soibelman.
\newblock Stability structures, motivic {Donaldson-Thomas} invariants and
  cluster transformations.
\newblock arXiv:0811.2435.

\bibitem[KT08]{kassel_turaev}
Christian {Kassel} and Vladimir {Turaev}.
\newblock {\em {Braid groups. With the graphical assistance of Olivier
  Dodane.}}, volume 247.
\newblock New York, NY: Springer, 2008.

\bibitem[NR13]{ng_rutherford}
Lenhard {Ng} and Daniel {Rutherford}.
\newblock {Satellites of Legendrian knots and representations of the
  Chekanov-Eliashberg algebra.}
\newblock {\em {Algebr. Geom. Topol.}}, 13(5):3047--3097, 2013.

\bibitem[PC05]{cp_4conj}
P.~E. {Pushkar } and Yu.~V. {Chekanov}.
\newblock {Combinatorics of fronts of Legendrian links, and Arnol'd's
  4-conjectures.}
\newblock {\em {Russ. Math. Surv.}}, 60(1):95--149, 2005.

\bibitem[{Rut}06]{rutherford06}
Dan {Rutherford}.
\newblock {The Thurston-Bennequin number, Kauffman polynomial, and ruling
  invariants of a Legendrian link: the Fuchs conjecture and beyond.}
\newblock {\em {Int. Math. Res. Not.}}, 2006(9):15, 2006.

\bibitem[Su18]{taosu}
Tao Su.
\newblock {\em A Hodge-theoretic study of augmentation varieties associated to
  Legendrian knots/tangles}.
\newblock PhD thesis, Berkeley, 2018.

\bibitem[{Toe}06]{toen_derived_hall}
Bertrand {Toen}.
\newblock {Derived Hall algebras.}
\newblock {\em {Duke Math. J.}}, 135(3):587--615, 2006.

\bibitem[{Tur}90]{turaev90}
V.~G. {Turaev}.
\newblock {Operator invariants of tangles, and R-matrices.}
\newblock {\em {Math. USSR, Izv.}}, 35(2):411--444, 1990.

\end{thebibliography}

\Addresses

\end{document}